\documentclass[12pt,reqno]{amsart}
\usepackage{amssymb}
\usepackage{amsmath}
\usepackage{versions}
\usepackage[usenames,dvipsnames] {color}

\def\gam{\Gamma}
\def\d{\mathbb{D}}
\def\c{\mathbb{C}}
\def\t{\mathbb{T}}

\def\ctwo{\mathbb{C}^2}

\def\royal{\mathcal{R}}
\def\rnk{\mathcal{R}^{n,k}}

\def\be{\begin{equation}}
\def\ee{\end{equation}}

\def\set#1#2{\{ #1 \, : \, #2\}}

\def\gaminn{\Gamma\text{-inner}}
\def\s0{s_0}
\def\p0{p_0}
\let\phi\varphi
\let\epsilon\varepsilon
\DeclareMathOperator{\ran}{ran}

\newtheorem{theorem}{Theorem}[section]
\newtheorem{corollary}[theorem]{Corollary}
\newtheorem{lemma}[theorem]{Lemma}
\newtheorem{proposition}[theorem]{Proposition}

\numberwithin{equation}{section}
\newtheorem{defin}[theorem]{Definition}
\newtheorem{lem}[theorem]{Lemma}
\newtheorem{prop}[theorem]{Proposition}
\newtheorem{thm}[theorem]{Theorem}

\newtheorem{problem}[theorem]{Problem}

\newtheorem{example}[theorem]{Example}

\newtheorem{definition}[theorem]{Definition}

\newtheorem{remark}[theorem]{Remark}

\newtheorem{fact*}[theorem]{Fact}
\DeclareMathOperator\hol{Hol}

\DeclareMathOperator\rank{rank}
\DeclareMathOperator\re{Re}

\newcommand\e{\mathrm{e}}

\newcommand\ii{\mathrm{i}}

\newcommand{\M}{b\Gamma}

\newcommand{\T}{\mathbb{T}}

\newcommand{\D}{\mathbb{D}}
\newcommand{\C}{\mathbb{C}}
\newcommand{\G}{\mathcal{G}}

\newcommand{\R}{\mathbb{R}}

\newcommand{\schur}{\mathcal{S}}
\newcommand \z{\mathbb Z}

\newcommand{\ip}[2]{\left\langle #1, #2 \right\rangle}

\newcommand{\twopartdef}[4]
{
	\left\{
		\begin{array}{ll}
			#1 &  #2 \\  \\
			#3 &  #4
		\end{array}
	\right.
}

\newcommand{\ph}{\varphi}
\newcommand\al{\alpha}

\newcommand\Ga{\Gamma}

\newcommand\la{\lambda}
\newcommand\ups{\upsilon}
\newcommand\si{\sigma}
\newcommand\ta{\theta}

\newcommand\half{{\tfrac{1}{2}}}

\newcommand\df{\stackrel{\rm def}{=}}
\newcommand\Aut{\mathrm{Aut}~}

\newcommand\up{\upsilon}
\newcommand\beq{\begin{equation}}
\newcommand\ds{\displaystyle}
\newcommand\eeq{\end{equation}}

\newcommand\nn{\nonumber}
\newcommand\bbm{\begin{bmatrix}}
\newcommand\ebm{\end{bmatrix}}
\newcommand\bpm{\begin{pmatrix}}
\newcommand\epm{\end{pmatrix}}

\def\qed{\hfill $\square$ \vspace{3mm}}
\let\phi\varphi
\numberwithin{equation}{section}
\begin{document}
\title[The construction of rational $\Gamma$-inner functions]{ Finite Blaschke products and  the construction of rational $\Gamma$-inner functions}
\author{Jim Agler, Zinaida A. Lykova and N. J. Young}
\date{submitted 13th April 2016, revised 30th October 2016}

\begin{abstract} 
Let 
\[
\Gamma \stackrel{\rm def}{=} \{(z+w, zw): |z|\leq 1, |w|\leq 1\} \subset \mathbb{C}^2.
\]
A {\em $\Gamma$-inner function} is a holomorphic map $h$ from the unit disc $\mathbb{D}$ to $\Gamma$ whose boundary values at almost all points of the unit circle $\mathbb{T}$ belong to the distinguished boundary $b\Gamma$ of $\Gamma$.
A rational $\Gamma$-inner function $h$ induces a continuous map $h|_\mathbb{T}$ from $\mathbb{T}$ to $b\Gamma$.  The latter set is topologically a M\"obius band and so has fundamental group $\mathbb{Z}$. The {\em degree} of $h$ is defined to be the topological degree of $h|_\mathbb{T}$.
In a previous paper the authors showed that if $h=(s,p)$ is a rational $\Gamma$-inner function of degree $n$ then $s^2-4p$ has exactly $n$ zeros in the closed unit disc $\mathbb{D}^-$, counted with an appropriate notion of multiplicity. In this paper,  with the aid of a solution of an  interpolation problem for finite Blaschke products, we explicitly construct the rational $\Gamma$-inner functions of degree $n$ with the $n$ zeros of $s^2-4p$ prescribed. 
\end{abstract}



\thanks{The first author was partially supported by the National Science Foundation Grant on  Extending Hilbert Space Operators DMS 1361720. 
The second and third authors were partially supported by the UK Engineering and Physical Sciences Research Council (EPSRC) grant EP/N03242X/1.
The third author was partially supported by EPSRC grant EP/K50340X/1. The collaboration was partially supported by London Mathematical Society Grant 41431.}

\maketitle
\section{Introduction}\label{Nintro}
The {\em symmetrized bidisc} is the set
\[
\Ga \df \{(z+w, zw): |z|\leq 1, |w|\leq 1\} \subset \c^2.
\]
$\Ga$ has attracted considerable interest in recent years because of its rich function theory \cite{biswas,ALY12,KZ}, complex geometry \cite{costara,edigarian,jarnicki,NiPfZw,NiPa,PZ}, some associated operator theory \cite{AY,AY99,bhattacharyya,bhattacharyya2,pal,pal2,sarkar} and its connection with the difficult problem of $\mu$-synthesis \cite{Ber03,ALY13,NJY11}.
The distinguished boundary of $\Ga$, that is, the \v Silov boundary of the algebra of continuous functions on $\Ga$ that are holomorphic in the interior of $\Ga$, will be denoted by $b\Ga$.
Concretely, $b\Ga$ is the symmetrization of the $2$-torus \cite[Theorem 2.4]{AY04}:
\[
b\Ga=\{(z+w, zw): |z|=|w|= 1\}.
\]
A {\em $\Gamma$-inner function} is a holomorphic map $h$ from the unit disc $\d$ to $\Ga$
whose boundary values at almost all points of the unit circle $\t$ (with respect to Lebesgue measure) belong to $b\Ga$.  The
$\Ga$-inner functions constitute a natural analog of the {\em inner functions } of A. Beurling \cite{beurling}, which play a central role in the function theory of the unit disc.  
For example, it was known to Nevanlinna and Pick that an $n$-point interpolation problem for functions in the Schur class is solvable if and only if it is solvable by a rational inner function of degree at most $n$.  Likewise, every $n$-point interpolation problem for functions in the class $\hol(\d,\Ga)$ of holomorphic maps from $\d$ to $\Ga$, if solvable, has a rational $\Ga$-inner solution (for example, \cite[Theorem 4.2]{Cost05}).  Here, the {\em degree} of a rational $\Ga$-inner function $h$ is defined to be the topological degree of the restriction of $h$ mapping $\t$ continuously to $b\Ga$.  Since $b\Ga$ is homeomorphic to a M\"obius band, its fundamental group is $\z$, and so the degree of $h$ is an integer; it will be denoted by $\deg(h)$.

We shall address the analog for rational $\Ga$-inner functions of a problem about rational inner functions solved by W. Blaschke \cite{blaschke}.  The Argument Principle tells us that a rational inner function $\ph$ of degree $n$ has exactly $n$ zeros in $\d$, counted with multiplicity,
from which fact one deduces that $\ph$ is a finite Blaschke product
\[
\ph(\la)= c\prod_{j=1}^n\frac{\la-\al_j}{1-\bar\al_j\la}
\]
where $|c|=1$ and $\al_1,\dots,\al_n$ are the zeros of $\ph$.  In similar fashion, we should like to write down, as explicitly as possible, the general rational $\Ga$-inner function of degree $n$.   It was shown in \cite{ALY14} that if  $h=(s,p)$ is a rational $\Ga$-inner function of degree $n$ then $s^2-4p$ has exactly $n$ zeros in the closed unit disc $\d^-$, counted with an appropriate notion of multiplicity. The $n$ zeros of $s^2-4p$ can be regarded as analogs of the $\al_j$ for present purposes.

The variety 
\begin{align}\label{eq1.30_I}
\royal &\df \set{(2z,z^2)}{z\in\C} \notag\\
	&=\{(s,p)\in\c^2: s^2=4p\}
\end{align}
plays a special role in the function theory of $\Ga$: it is called the {\em royal variety}.
For a rational $\Ga$-inner function $h=(s,p)$, the zeros of $s^2-4p$ in $\d^-$ are the points $\la$ such that $h(\la)\in\royal$; we shall call them the
{\em royal nodes} of $h$.  If $\si\in\d^-$ is a royal node of $h$ then $h(\si)=(-2\eta,\eta^2)$ for some $\eta\in\d^-$; we call $\eta$ the {\em royal value} of $h$ corresponding to the royal node $\si$.

 Let us formalise the problem of describing the general rational $\Ga$-inner function in terms of its royal nodes and values.  
\begin{problem}\label{prob1}
 Given distinct points $\si_1,\dots,\si_n$ in $\d^-$ and values $\eta_1,\dots,\eta_n$ in $\d^-$ find if possible a rational $\Ga$-inner function $h$ of degree $n$ such that
\begin{align*}
h(\sigma_j) &= (-2\eta_j,\eta_j^2) \quad \text{ for } j=1, \ldots, n.
\end{align*}
\end{problem}

The results of this paper show that there is a close connection between Problem \ref{prob1} and  an $n$-point interpolation problem for finite Blaschke products of degree $n$ in which there are interpolation nodes in both $\d$ and $\t$ and in which tangential information is specified at interpolation nodes in $\t$.  To formulate this problem we introduce some terminology.
\begin{defin}\label{blaschkedata}
Let $n \ge 1$ and $0 \le k \le n$.   By {\em Blaschke interpolation data} we mean a triple $(\si,\eta,\rho)$ where
\begin{enumerate}
\item  $\sigma = (\sigma_1, \sigma_2, \ldots, \sigma_n)$ is an $n$-tuple of distinct points such that  $\sigma_j \in \t$ for $j=1,\dots, k$ and $\sigma_j \in \d$ for $j=k+1,\dots, n$;
\item   $\eta = (\eta_1, \eta_2, \ldots, \eta_n)$ where $\eta_j \in \t$ for $j=1,\dots, k$ and $\eta_j \in \d$ for $j=k+1,\dots,  n$;
\item  $\rho = (\rho_1, \rho_2, \ldots, \rho_k)$ where $\rho_j > 0$ for $j=1,\dots,k$.
\end{enumerate}
\end{defin}
For such data the {\em  Blaschke interpolation problem} with data $(\si,\eta,\rho)$ is the following:
\begin{problem}\label{prob}
Find if possible a rational inner function $\phi$ on $\d$ (that is,  a finite Blaschke product) of degree $n$ with the properties
\begin{align}
\phi(\sigma_j) &= \eta_j \quad \text{ for } j=1, \ldots, n \label{eq6.2}
\end{align}
and
\begin{align}
A\phi(\sigma_j)&=\rho_j \quad \text{ for } j=1, \ldots, k,  \label{eq6.4}
\end{align}
where $A\ph(\e^{i\theta})$ denotes the rate of change of the argument of $\ph(\e^{i\theta})$ with respect to $\theta$.
\end{problem}

Problem \ref{prob} has been much studied, for example \cite{RR,Sar,BH, jr87,gl2002,sw2006,gr2008}.  Without the tangential conditions
\eqref{eq6.4}, or some other constraint (for example, a degree constraint), the problem would arguably be ill-posed: solvability would depend only on the interpolation conditions at nodes in $\d$, and the conditions at $\si_1,\ldots,\si_k$ would be irrelevant.  With the conditions
\eqref{eq6.4}, however, the problem has an elegant solution.
There is a simple criterion for the existence of a solution of Problem \ref{prob} in terms of an associated ``Pick matrix'', and better still, there is an explicit parametrization of all solutions $\ph$ by a linear fractional expression in terms of a parameter $\zeta\in\t$.  There are polynomials $a,b,c$ and $d$ of degree at most $n$ such that the general solution of Problem \ref{prob} is
\be\label{lnfrzeta}
\ph=\frac{a\zeta+b}{c\zeta+d}
\ee
where the parameter $\zeta$ ranges over a cofinite subset of $\t$
(see Theorem \ref{refSarason} below).  The polynomials $a,b,c$ and $d$ are unique subject to a certain normalization.

Analogously, Problem \ref{prob1} needs to be modified by the addition of tangential conditions at interpolation nodes in $\t$ in order to be well posed.  We are led to the following refinement of Problem \ref{prob1}. 
\begin{problem}\label{royalinterp}
 Given Blaschke interpolation data $(\si,\eta,\rho)$ with $n$ interpolation nodes of which $k$ lie in $\t$, find if possible a rational $\Ga$-inner function $h=(s,p)$ of degree $n$ such that
\begin{align*}
h(\si_j)&=(-2\eta_j,\eta_j^2) \quad \mbox{ for }j=1,\dots,n
\end{align*}
and
\[
Ap(\si_j)=2\rho_j \quad \mbox{ for } j=1,\dots,k.
\]
\end{problem}
We shall call this the {\em royal $\Ga$-interpolation problem} with data $(\si,\eta,\rho)$.

The connection between Problems \ref{royalinterp} and \ref{prob} can be described with the aid of a certain $1$-parameter family of rational functions $\Phi_\omega$ on $\Ga$, where $\omega\in\t$. These functions play a central role in the function theory of $\Ga$ (for example, \cite{AY,AY04}).  They are defined by 
\be\label{eq1.26_I}
\Phi_\omega(s,p) = \frac{2\omega p - s}{2 - \omega s}.
\ee
$\Phi_\omega$ is holomorphic on $\Ga$, except for a singularity at $(2\bar\omega,\bar\omega^2)$, and maps $\Ga$ into $\d^-$.  They constitute a universal set of Carath\'eodory extremal functions for the interior of $\Ga$ \cite[Corollary 3.4]{AY04}. 

A consequence of Theorems \ref{thm7.10} and \ref{thm7.20} is:
\begin{theorem}
For Blaschke interpolation data $(\si, \eta,\rho)$  the following two statements are equivalent
\begin{enumerate}
\item Problem {\rm \ref{royalinterp}} with data $(\si, \eta,\rho)$ is solvable by a rational $\Ga$-inner function $h$ such that $h(\d) \not\subset \royal$;
\item Problem {\rm \ref{prob}} with data $(\si, \eta,\rho)$  is solvable
and there exist $s_0,p_0\in\c$ such that
\begin{align*}
|s_0|&<2, \quad |p_0|=1,\\
s_0&=\bar s_0p_0,
\end{align*}
\[
s_0a-2b+2p_0c-s_0d=0,
\]
where $a,b,c$ and $d$ are the polynomials in the normalized parametrization \eqref{lnfrzeta} of the solutions of Problem {\rm \ref{prob}}.
\end{enumerate}
\end{theorem}
The $\Ga$-inner functions whose range is contained in $\royal$, those of the form $(2f,f^2)$ for some inner $f$, behave differently from others.

Theorem \ref{thm7.20} gives a formula for a solution $h$ of Problem \ref{royalinterp} in terms of $s_0,p_0,a,b,c$ and $d$.  Since the polynomials $a,b,c$ and $d$ are computed in Theorem \ref{thm6.10} and Remark \ref{explicit-solution}, we obtain an explicit solution of 
Problem \ref{royalinterp}. The algorithm is presented in Section \ref{algorithm}.

The connection between the solution sets of the royal $\Ga$-interpolation problem and the Blaschke
interpolation problem can be made explicit with the aid of the functions $\Phi_\omega$.
\begin{theorem}\label{thm1.1}
Let $(\si, \eta,\rho)$ be Blaschke interpolation data.  Suppose that $h$ is a solution of Problem {\rm \ref{royalinterp}} with these data and that $h(\d) \not\subset \royal$.
For all
$\omega\in\t\setminus\{-\bar\eta_1,-\bar\eta_2,\ldots,-\bar\eta_k\}$, the function $\Phi_\omega\circ h$ is a solution of Problem {\rm \ref{prob}}  with the same data.  
Conversely, for every solution $\ph$ of the Blaschke interpolation problem with data $(\si, \eta,\rho)$, there exists $\omega\in\t$ such that $\ph=\Phi_\omega\circ h$.
\end{theorem}

This theorem is a corollary to Theorem \ref{thm7.10}.

In an earlier paper \cite{ALY14} the authors gave another construction of the general rational $\Ga$-inner function $h=(s,p)$ of degree $n$, starting from different data, to wit, the royal nodes of $h$ and the zeros of $s$.  One step in the construction in \cite{ALY14} is to perform a Fej\'er-Riesz factorization of a non-negative trigonometric polynomial, whereas, in contrast, the construction in this paper can be carried out entirely in rational arithmetic.


\section{Background material}\label{background}

In this section we establish some notation and terminology and present some
elementary facts about the set $\Ga$ discussed in the introduction.

The following results afford useful criteria for membership of $\gam$ and $\M$ \cite{AY04}.
\begin{proposition}\label{prop2.10}
Let $(s,p) \in \ctwo$.  The point $(s,p)$ lies in $ \gam$ if and only if
\be\label{eq2.30}
|s| \le 2 \ and\ |s-\overline{s}p| \le 1-|p|^2. \notag
\ee
The point $(s,p)$ lies in $\M$ if and only if
\be\label{eq2.40}
|s| \le 2,\ |p| =1,\ and\ s-\overline{s}p = 0. \notag
\ee
\end{proposition}

The interior of $\Ga$, the {\em open  symmetrized bidisc} 
\beq\label{defG}
\G\df \set{(z+w,z w)}{|z| < 1,|w| < 1}
\eeq
will also arise.

 Proposition \ref{prop2.10} implies that if  $h=(s,p) \in \hol(\D,\c^2)$  then $h$ is $\gaminn$ if and only if $p$ is inner, $|s|$ is bounded by 2 on $\d$ and $s(\tau)-\overline{s(\tau)}p(\tau) = 0$ for almost all $\tau \in \t$ with respect to Lebesgue measure (by Fatou's theorem, $s$ and $p$ have non-tangential limits a.e. on $ \t$). This paper focuses on the case that $h$ is rational (that is, $s$ and $p$ are rational), in which case $s=\bar s p$ on the whole of $\t$. 

 Let us clarify the notion of the degree of a rational
$\Gamma$-inner function $h$. 
\begin{definition}
 The degree $\deg(h)$ of a rational $\Ga$-inner function $h$ is defined to be  $ h_*(1)$, where 
$h_* : \mathbb{Z}= \pi_1(\T) \to \pi_1(\M)$ is the homomorphism of fundamental groups induced by $h$ when it is regarded as a continuous map from $\T$ to $\M$.
\end{definition}
According to \cite[Proposition 3.3]{ALY14}, for any rational $\Gamma$-inner function $h= (s,p)$, $\deg (h)$ is equal to the degree $\deg (p)$ (in the usual sense) of the finite Blaschke product $p$.

We denote by $\schur$ the {\em Schur class}, which comprises all holomorphic maps from $\d$ to $\d^-$.

\begin{definition}\label{phasar-derivative}
 For any differentiable function $f: \T \to \C \setminus \{0\}$
 the {\em phasar derivative} of $f$ at $z= \e^{\ii \ta} \in \T$
is  the derivative with respect to $\ta$ of the argument of $f(\e^{\ii \ta})$ at $z$; we denote it by $Af(z)$.
\end{definition}
Thus, if $f(\e^{\ii \ta}) = R(\ta) \e^{\ii g(\ta)}$ is differentiable,   where $g(\ta) \in\R$ and $R(\ta) >0$, then $g$ is differentiable on $[0,2\pi)$ and the phasar derivative of $f$ at $z= \e^{\ii \ta} \in \T$ is equal to
\begin{equation}\label{phasar-deriv-arg}
Af(\e^{\ii \ta}) = \frac{d}{d \ta} \arg f(\e^{\ii \ta}) = g'(\ta).
\end{equation}
The above is not standard notation, but we shall find it useful in the sequel.  We summarise some
elementary properties of phasar derivatives.
\begin{proposition}\label{elphasar}
\begin{enumerate}
\item For differentiable functions $\psi, \phi : \T \to \C \setminus \{0\}$ and
for any $c \in \C \setminus \{0\}$, 
\beq\label {phasar-deriv-sum}
A(\psi \phi ) =A \psi +A \phi \; \quad \text{and} \; \quad A(c \psi  ) =A \psi.
\eeq

\item   For any rational inner function $\phi$ and for all $z \in \T$,
\begin{equation} \label{phasar-deriv-inner}
 A\phi(z)= z \frac{\phi'(z)}{\phi(z)}.
\end{equation}
\item If $\al\in\D$ and
$$
B_\alpha(z) = \frac{z-\alpha}{1-\overline \alpha z}
$$
then
\[
A B_\alpha(z)=\frac{1-|\al|^2}{|z-\al|^2} >0 \quad \mbox{ for }z\in \T.
\]
\item For any rational inner function $p$,
$$
A p (z) >0 \quad \mbox{ for  all } z \in \T.
$$
\end{enumerate}
\end{proposition}

Recall that  a point $\la \in \d^-$ is a  royal node of  a $\Gamma$-inner function $h$ if and only if  $h(\la)$ is in the royal variety $\royal = \{(2z, z^2): z \in \C \}$.

In the next proposition we shall use the notation $\Phi(z,s,p)$ as a synonym for $\Phi_z(s,p)$.  Thus, for any function $\ups$ on $\d$,
\[
\Phi \circ (\up, h) = \frac{2 \up p -s}{2- \up s}. 
\]
\begin{proposition}\label{phaser_h} 
Let $h =(s,p)$ be a rational $\Gamma$-inner function and let $\sigma$ be a royal node of $h$ on $\T$.
Then
 
{\rm (i)} there exists $\eta \in \T$ such that $p(\sigma)= \eta^2$ and $s(\sigma) = -2 \eta$;

{\rm (ii)} $\sigma$ is a  zero of $s^2-4p$ of multiplicity at least $2$;

{\rm (iii)} for any finite Blaschke product $\up$,
$$\Phi \circ (\up, h)(\sigma) = \eta$$
independent of $\up$;

{\rm (iv)} for any finite Blaschke product $\up$ such that 
$\up(\sigma) \neq - \bar{\eta}$, 
\[
A\Phi \circ (\up, h)(\sigma) = \half A p(\sigma).
\]
\end{proposition}
\begin{proof}
(i) By \cite[Lemma 7.10]{ALY12}, the royal nodes of $h =(s,p)$ on $\T$ are precisely the points $\sigma \in \T$ such that $|s(\sigma) | =2$. Thus there exists $\eta \in \T$ such that
$s(\sigma) = -2 \eta$ and, since $ 4p(\sigma)= s(\sigma)^2$,  we have $p(\sigma)= \eta^2$.

(ii) By Proposition \ref{prop2.10}, on $\T$ we have
\[
\bar{p}(4 p -s^2) =4 -(\bar{p} s) s = 4- \bar{s}s = 4- |s|^2 \ge 0.
\] 
Since $|s(\sigma) | =2$, the function $f(\theta) = 4 -|s(e ^{i \theta}) |^2$ has a local minimum at $\xi$ where $\sigma =e^{i \xi}$. Therefore
\begin{eqnarray}
\label {deriv-1}
0 &=& \frac{d}{d \theta}\left( 4 - |s(e ^{i \theta}) |^2 \right)_{|_{\xi}} \nn\\
&=& ~\frac{d}{d \theta}
\left(\bar{p}(4 p -s^2)(e ^{i \theta})\right)_{|_{\xi}} \nn\\
&=&p(e^{i \xi}) i e^{i \theta}(4 p'(e^{i \xi})- 2 s s'(e^{i \xi})).
\end{eqnarray}
Hence
\[
(4p -s^2)'(\sigma)=0,
\]
and so $\sigma$ is a zero of $s^2 -4p$ of multiplicity at least $2$.

(iii) If $\up s (\sigma) \neq 2$, then
\begin{eqnarray}
\label {phi_sigma}
\Phi \circ (\up, h)(\sigma)&=&  \frac{2 \up p -s}{2- \up s} (\sigma)\nn\\
&=& \frac{2 \up \frac{1}{4} s^2 -s}{2- \up s} (\sigma) \nn\\
&=& \frac{s (\up \half s -1)}{2(1- \half \up s)} (\sigma) \nn\\
&=& -\half s(\sigma) = \eta.
\end{eqnarray}
Thus
$\Phi \circ (\up, h)(\sigma) = \eta$
independent of $\up$, as long as $\up s (\sigma) \neq 2$, that is, $\up(\sigma) \neq - \bar{\eta}$.

 For any finite Blaschke product $\up$ such that 
$\up(\sigma) = - \bar{\eta}$, by Proposition \ref{elphasar}, we have
$$ Ap(\sigma) = \sigma \frac{p'(\sigma)}{p(\sigma)}= \sigma \bar{\eta}^2 p'(\sigma)$$
and 
$$ A \up(\sigma) = \sigma \frac{\up'(\sigma)}{\up(\sigma)}= - \sigma \frac{\up'(\sigma)}{\bar{\eta}}.$$
Since $\up$ and $p$ are inner functions $$ A \up(\sigma) \neq  - \half Ap(\sigma),$$
which is equivalent to
$$\up'(\sigma) \neq \half \bar{\eta}^3 p'(\sigma).$$
Note that $\up(\sigma) = - \bar{\eta}$ implies that
$$ 2 \up(\sigma) p(\sigma) -s(\sigma) =0 = 2 - \up(\sigma)s(\sigma),
$$
and so
\begin{eqnarray}
\label {phi_sigma_2}
\Phi \circ (\up, h)(\sigma)&=&  \frac{(2 \up p -s)'}{(2- \up s)'} (\sigma)\nn\\
&=& \frac{2 \up' p+ 2 \up p' -s'}{- \up' s- \up s'} (\sigma) \nn\\
&=& \frac{2 \up' \eta^2+ 2 (-\bar{\eta}) p' +\bar{\eta} p'}{- \up' (-2 \eta)- (-\bar{\eta})(-\bar{\eta} p')} (\sigma) \nn\\
&=& \frac{2 \eta^2 \up'-\bar{\eta} p'}{2 \eta \up'-\bar{\eta}^2 p'} (\sigma) \nn\\
&=& \eta \frac{2 \eta \up'-\bar{\eta}^2 p'}{2 \eta \up'-\bar{\eta}^2 p'}(\sigma) = \eta.
\end{eqnarray}
Thus
$\Phi \circ (\up, h)(\sigma) = \eta$
independent of $\up$.

(iv)  For any finite Blaschke product $\up$ such that 
$\up(\sigma) \neq - \bar{\eta}$, by Proposition \ref{elphasar}, we have
\begin{eqnarray}\label {Phasar_sigma}
A\Phi \circ (\up, h)(\sigma) &=& A(2 \up p -s)(\sigma)  - A(2- \up s) (\sigma) \nn\\
 &=& \sigma \frac{2 \up' p +2 \up p' - s'}{2 \up p -s}(\sigma) - \sigma \frac{-\up' s - \up s'}{2- \up s}(\sigma)\nn\\
 &=&  \sigma \frac{2 \up' \eta^2 +2 \up p' +\bar{\eta} p'}{2 \up \eta^2 +2 \eta}(\sigma) + \sigma \frac{\up' (-2 \eta) + \up (-\bar{\eta} p')}{2+ 2\up \eta}(\sigma)\nn\\
 &=& \frac{\sigma}{2+ 2\up \eta} \left(2 \up' \eta +2 \bar{\eta} \up p' +\bar{\eta}^2 p' - 2\up' \eta - \up \bar{\eta} p' \right)(\sigma)\nn\\
&=& \frac{\sigma \bar{\eta}^2p'(\sigma)(1+ \up \eta)}{2(1+ \up \eta)}\nn\\
&=&\half \frac{\sigma p'(\sigma)}{ p(\sigma)}\nn\\
&=&\half A p(\sigma).
\end{eqnarray}
\qed
\end{proof}

\section{The Blaschke interpolation problem and rational $\gaminn$ functions}\label{mixed}
The  Blaschke interpolation problem, Problem \ref{prob}, is  an algebraic variant of the classical Pick interpolation problem. One seeks a Blaschke product of a given degree $n$ satisfying $n$ interpolation conditions, rather than merely a Schur-class function, and one admits interpolation nodes in both the open unit disc and the unit circle.  As with the classical Nevanlinna-Pick problem, there is a criterion for the solvability of such a problem in terms of the positivity of a `Pick matrix' formed from the interpolation data; however, to obtain a concise formulation, one has to impose additional interpolation conditions, on phasar derivatives at the interpolation nodes on the circle, and the bounds on these phasar derivatives appear on the diagonal entries of the Pick matrix.  This modified Pick matrix appears in the work of several authors \cite{bgr,AgMcC,Sar,geo},  but for simplicity we shall continue to speak of the Pick matrix. To be precise, the \emph{Pick matrix} associated with Blaschke interpolation data $(\si,\eta,\rho)$ as in Definition \ref{blaschkedata} is defined to be the $n \times n$ matrix $M=[m_{ij}]_{i,j=1}^n$ with entries
$$
m_{ij} = \twopartdef{\rho_i}{\mbox{ if }i=j\leq k}{\ds \frac{1-\overline{\eta_i}\eta_j}{1-\overline{\sigma_i}\sigma_j}}{\mbox{ otherwise}.}
$$
\begin{remark} {\rm
Of course, it can happen for $n$-point Blaschke interpolation data $(\si,\eta,\rho)$ that there exists a  Blaschke product $\phi$ of degree {\em strictly} less than $n$  satisfying the conditions \eqref{eq6.2} to \eqref{eq6.4}, but in the present context we are concerned with solutions of degree exactly $n$.
}\end{remark}

In the case that $n=k$, that is, where all the interpolation nodes lie on the unit circle there is an elegant solvability criterion due to D. Sarason \cite{Sar}.
His result implies that, when $n=k$,  Problem \ref{prob} is solvable if and only if the corresponding Pick matrix $M$ is minimally positive, that is, when $M\geq 0$ and there is no positive diagonal $n \times n$ matrix $D$, other than $D =0$, such that $M \ge D$.
Actually, Sarason considers interpolation by functions in the Schur class, not just Blaschke products, and so there is a subtlety concerning the existence of phasar derivatives at boundary points (related to the Julia-Carath\'eodory theorem), but since we are only concerned with rational functions, no such difficulty will arise here.

 The following result is well known -- see \cite[Sections 21.1 and 21.4]{bgr} or \cite{AgMcC,Sar,geo}.
\begin{prop}\label{M>0}
If Problem {\rm \ref{prob}} is solvable then the corresponding Pick matrix $M$ is positive definite and the solution of the problem is not unique.
\end{prop}
 Several authors have developed deep and far-reaching machines to characterise solvability of interpolation problems for classes related to Problem \ref{prob}, and to parametrize their sets of solutions \cite{bgr,BH,BD,kh1998,ChH,geo};  there is a brief history in \cite[Notes for Part V, page 500]{bgr}. 
A paper which addresses the combined interior and boundary problem specifically for finite Blaschke products is \cite{gl2002}.
  However, we have not found the precise statement that we need, and so, for the convenience of the reader, we give a self-contained treatment.

Our strategy for the construction of the general solution of Problem {\rm \ref{prob}} is to adjoin an additional boundary interpolation condition; this augmented problem will have a unique solution, and in this way we obtain all solutions of Problem {\rm \ref{prob}} in terms of a unimodular parameter.

The following is a refinement of the Sarason Interpolation Theorem \cite{Sar}, in that we consider interpolation nodes both on the circle and in the open disc.  The result is contained in \cite[Theorem 2.5]{bol2011}.   See also \cite[Theorem 5.2]{ChH} for a solution to the analogous interpolation problem for the upper half plane.
\begin{theorem}\label{refSarason} 
Let $M$ be the Pick matrix associated with Blaschke interpolation data $(\si,\eta,\rho)$. 
\begin{enumerate}
\item There exists a function $\phi$ in the Schur class such that 
\begin{align}
\phi(\sigma_j) &= \eta_j \quad \text{ for } j=1, \ldots, n, \label{eq6.2_2} 
\end{align}
and the phasar derivative $A\ph(\si_j)$ exists 
and satisfies
\begin{align}
 A\phi(\sigma_j)&\leq\rho_j \quad \text{ for } j=1, \ldots, k  \label{eq6.4_2a}
\end{align}
if and only if $M\geq 0$;
\item  if $M$ is positive and of rank $r < n$ then there is a {\em unique} function $\ph$ in the Schur class satisfying conditions \eqref{eq6.2_2} and \eqref{eq6.4_2a}, and this function is a  Blaschke product of degree $r$;
\item   the unique function $\ph$ in statement {\rm (2)} satisfies
\begin{align}
 A\phi(\sigma_j)&=\rho_j \quad \text{ for } j=1, \ldots, k  \label{eq6.4_2}
\end{align}
if and only if $M$ is minimally positive.
\end{enumerate}
\end{theorem}

Consider a point $\tau \in \t$ distinct from $ \sigma_1,\dots, \sigma_k$. For each $\zeta \in \t$ we seek a solution $\phi$ to Problem {\rm \ref{prob}} that satisfies the additional interpolation condition $\phi(\tau)=\zeta$ and $A \phi(\tau) = \rho_{\zeta, \tau}$, where $\rho_{\zeta, \tau} >0$ is chosen to make the Pick matrix $B_{\zeta, \tau}$ of the augmented interpolation problem singular. We record the following simple lemma without proof.

\begin{lemma}\label{lem6.10}
If $C$ is an $n \times n$ positive definite matrix, $u$ is an $n \times 1$ column, $\rho=\langle C^{-1} u, u \rangle$ and the $(n+1) \times (n+1)$ matrix $B$ is defined by
$$B = \begin{bmatrix}C&u\\ u* &\rho\end{bmatrix},$$
then $B$ is positive semi-definite, $\rank(B) = n$ and
$$
B\begin{bmatrix}-C^{-1}u\\ 1\end{bmatrix} = 0.
$$
\end{lemma}

The Pick matrix $B_{\zeta, \tau}$ of the augmented problem is the $(n+1) \times (n+1)$ matrix,
\be\label{eq6.10}
B_{\zeta, \tau} = \begin{bmatrix}M&u_{\zeta, \tau}\\ u_{\zeta, \tau}^* &\rho_{\zeta, \tau}\end{bmatrix}
\ee
where $M$ is the Pick matrix associated with Problem {\rm \ref{prob}}, $u_{\zeta, \tau}$ is the $n \times 1$ column matrix defined by
\be\label{u-zeta-tau}
u_{\zeta, \tau} = \begin{bmatrix}\frac{1-\overline{\eta_1}\zeta}{1-\overline{\sigma_1}\tau}\\ \vdots\\ \frac{1-\overline{\eta_n}\zeta}{1-\overline{\sigma_n}\tau}\end{bmatrix}
\ee
and
$$
\rho_{\zeta, \tau} = \langle M^{-1} u_{\zeta, \tau}, u_{\zeta, \tau} \rangle.
$$

Thus the {\em augmented problem} that we are considering is the  Blaschke interpolation problem with data $(\tilde \si, \tilde\eta,\tilde\rho)$ where
\[
\tilde\si=(\si,\tau), \quad \tilde\eta = (\eta,\zeta), \quad \tilde\rho= (\rho, \rho_{\zeta, \tau}).
\]
\begin{prop}\label{rankM}
Let $\psi$ be a Blaschke product of degree $N$. Let 
 $\sigma = (\sigma_1, \sigma_2, \ldots, \sigma_n)$ be an $n$-tuple of distinct points in $\d^-$, let
$ \eta_j= \psi(\sigma_j)$ for $j=1,\dots, n$ and let
 $\rho_j = A\psi(\sigma_j)$ for $j$ such that $|\sigma_j|=1$.
The Pick matrix for the data $(\si, \eta,\rho)$ has rank at most
$N$.
\end{prop}
\begin{proof}
In the case that the $\sigma_j$ all lie in $\d$ the assertion is well known -- see \cite{AgMcC}.  It follows easily from the fact that in this case the Pick matrix $M$ is given by
\[
M=\bbm \ip{(1-T_\psi T_\psi^*)k_{\la_j}}{k_{\la_i}} \ebm_{i,j=1}^n,
\]
where $k_\la$ denotes the Szeg\H o kernel and $T_\psi$ is the analytic Toeplitz operator on the Hardy space $H^2$ with symbol $\psi$.

Consider the case that $\si_1,\ldots,\si_k\in\t$ and $\si_{k+1},\ldots,\si_n \in\d$. Let $M$ be the Pick matrix for the data $(\si,\eta,\rho)$. Choose $r\in(0,1)$ and let $\la_j=r\si_j$ for $j=1,\ldots,n$.  By the foregoing observation, the matrix
\[
M(r)\df   \bbm \displaystyle\frac{1-\overline{\psi(r\si_i)}\psi(r\si_j)}{1-r^2\overline{\si_i}\si_j} \ebm_{i,j=1}^n
\]
has rank at most $N$. Let $r\to 1-$.  It follows from L'H\^opital's rule that the $j$th diagonal entry, for $j=1,\ldots,k$, tends to $A\psi(\si_j)$.  The remaining entries of $M(r)$ also tend to the corresponding entries of $M$, and so $M(r) \to M$.  It follows that $\rank(M)\leq N$. \qed 

\end{proof}

\begin{prop}\label{atmost1}
If the Pick matrix $M$ associated with Problem {\rm \ref{prob}} is positive definite then, for any $\tau\in \t\setminus \{\si_1,\dots,\si_k\}$ and  $\zeta \in \t$ there is at most one solution $\ph$ of Problem {\rm \ref{prob}} for which $\ph(\tau)=\zeta$.
\end{prop}
\begin{proof}
Let $\psi$ be a solution of Problem \ref{prob} such that $\psi(\tau)=\zeta$ and let $\rho_\tau=A\psi(\tau)$.
Thus $\psi$ is in the Schur class and satisfies
\begin{align}\label{psicond}
\psi(\si_j)&= \eta_j\quad \mbox{ for } j=1,\dots,n, \notag\\
A\psi(\si_j) &= \rho_j \quad \mbox{ for }j=1,\dots,k,\notag\\
\psi(\tau)&=\zeta, \notag\\
A\psi(\tau)&= \rho_{\tau}.
\end{align}
Since $\psi$ is a Blaschke product of degree $n$, 
it follows from Proposition \ref{rankM}, applied to the augmented problem with data $(\tilde\si,\tilde\eta,(\rho,\rho_\tau))$, that
the corresponding Pick matrix
\[
\tilde M = \bbm M & u_{\zeta,\tau} \\ u_{\zeta,\tau}^* & \rho_\tau \ebm
\]
has rank less or equal to $ n$ and so it is singular.  Thus
\[
\rho_\tau= \langle M^{-1} u_{\zeta, \tau}, u_{\zeta, \tau} \rangle=\rho_{\zeta, \tau},
\]
 and so $A\psi(\tau)$ is the same for every solution of Problem \ref{prob} such that $\psi(\tau)=\zeta$.
By Theorem \ref{refSarason},
there is a {\em unique} function $\psi$ in the Schur class satisfying the conditions \eqref{psicond}, and hence there is at most one solution of Problem \ref{prob} such that $\psi(\tau)=\zeta$. \qed
\end{proof}

We denote by $e_j$ the $j$th standard basis vector in $\C^n$.
\begin{proposition}\label{prop6.10}
If the Pick matrix $M$ associated with Problem {\rm \ref{prob}} is positive definite, if $\tau\in\t\setminus\{\si_1,\dots,\si_k\}, \, \zeta \in \t$ and
\be\label{eq6.20}
\langle M^{-1}u_{\zeta, \tau} , e_j \rangle \ne 0
\ee
for $j=1,\ldots,k$, then there exists a unique solution $\phi$ to Problem {\rm \ref{prob}} such that $\phi(\tau)=\zeta$.
\end{proposition}
\begin{proof}
Observe that by Lemma \ref{lem6.10}, if $B_{\zeta, \tau}$ is defined by  equation \eqref{eq6.10}, then $B_{\zeta, \tau} \ge 0$ and $\rank(B_{\zeta, \tau})=n$. Further, Lemma \ref{lem6.10} guarantees that $\ker(B_{\zeta, \tau})$ is spanned by the vector $\begin{bmatrix}-M^{-1}u_{\zeta, \tau}\\ 1\end{bmatrix}$. The inequation \eqref{eq6.20} implies, for $j=1,\ldots,k$,  that $\begin{bmatrix}-M^{-1}u_{\zeta, \tau}\\ 1\end{bmatrix} \not\perp \begin{bmatrix}e_j\\ 0\end{bmatrix}$ and therefore, for every $\epsilon > 0$, we have
$$
\left\langle \left(B_{\zeta, \tau} - \epsilon\begin{bmatrix}e_j\\ 0\end{bmatrix} \otimes \begin{bmatrix}e_j\\ 0\end{bmatrix} \right)
\begin{bmatrix}-M^{-1}u_{\zeta, \tau}\\ 1\end{bmatrix}, \begin{bmatrix}-M^{-1}u_{\zeta, \tau}\\ 1\end{bmatrix} \right\rangle =
$$
$$
- \epsilon \left| \left\langle   \begin{bmatrix}-M^{-1}u_{\zeta, \tau}\\ 1\end{bmatrix} , \begin{bmatrix}e_j\\ 0\end{bmatrix} \right\rangle \right|^2 < 0.
$$
Thus
$$
B_{\zeta, \tau} - \epsilon\begin{bmatrix}e_j\\ 0\end{bmatrix} \otimes \begin{bmatrix}e_j\\ 0\end{bmatrix} \not\ge 0.
$$
It follows that $B_{\zeta, \tau}$ is  minimally positive and the proposition follows from Theorem \ref{refSarason}.
 \qed \end{proof}
In the light of Proposition \ref{prop6.10} we define the exceptional set $Z_\tau $  for Problem \ref{prob} to be 
\be\label{defZtau}
Z_\tau=\{\zeta \in \t: \;\text{ for some} \; j, \; 1 \le j \le k, \; \langle M^{-1}u_{\zeta, \tau} , e_j \rangle = 0\}.
\ee
Define $n \times 1 $ vectors $x_\lambda$ and $y_\lambda$
for $\lambda \in \d^- \setminus \{ \sigma_1, \dots, \sigma_k\}$ by the formulas
\be\label{x-y-lambda}
x_\lambda =\begin{bmatrix}\frac{1}{1-\overline{\sigma_1}\lambda}\\ \vdots\\ \frac{1}{1-\overline{\sigma_n}\lambda}\end{bmatrix}, \qquad y_\lambda =\begin{bmatrix}\frac{\overline{\eta_1}}{1-\overline{\sigma_1}\lambda}\\ \vdots\\ \frac{\overline{\eta_n}}{1-\overline{\sigma_n}\lambda}\end{bmatrix},
\ee
so that
\be\label{utauxtau}
u_{\zeta, \tau} = x_\tau - \zeta y_\tau .
\ee

\begin{proposition}\label{Ztau} {\rm (i)} For any $\tau\in \t \setminus \{\sigma_1, \dots, \sigma_k \}$ 
if
$$
\langle x_\tau , M^{-1}e_j \rangle =0 = \langle y_\tau ,M^{-1} e_j \rangle\; \text{ for some } \; j,\; 1\le j \le k, 
$$ 
then  $Z_\tau = \t$.

{\rm (ii)} There exist uncountably many $\tau\in \t \setminus \{\sigma_1, \dots, \sigma_k \}$ such that the equation
$$
\langle x_\tau, M^{-1}e_j \rangle = 0= \langle y_\tau , M^{-1}e_j \rangle 
$$
does not hold  for any $j$, $1\le j \le k$.
Moreover, for such $\tau$, the set $Z_\tau$ consists of at most $k$ points.
\end{proposition}
\begin{proof} (i) 
Let 
$$
Z^j_\tau=\{\zeta \in \t: \;(x_\tau - \zeta y_\tau) \bot M^{-1} e_j\};
$$
By the definition \eqref{defZtau} and equation \eqref{utauxtau},
$$ 
Z_\tau=Z^1_\tau \cup \dots\cup Z^k_\tau.
$$
Note that
$ Z^j_\tau=\t$
if and only if, for every  $\zeta \in \t$,
$$\langle x_\tau - \zeta y_\tau,M^{-1} e_j \rangle = \langle x_\tau,M^{-1} e_j \rangle - \zeta \langle y_\tau,M^{-1} e_j \rangle =0.$$
Hence, for $\tau \in \t \setminus \{\sigma_1, \dots, \sigma_k \} $, the set 
$ Z^j_\tau=\t$ if and only if 
$$\langle x_\tau , M^{-1}e_j \rangle =0 = \langle y_\tau , M^{-1}e_j \rangle.
$$ 
Otherwise, $ Z^j_\tau$ consists of at most one point $\zeta^j_\tau \in \t$. \\

We shall call a point $\tau \in \t\setminus \{\sigma_1, \dots, \sigma_k \}$ {\em unsuitable} if there exists $j$, $1\le j \le k$, such that $ M^{-1} e_j \bot \{x_\tau, y_\tau \} $. \\

(ii) For $j$, $1\le j \le k$, let
$$
E_j = \{\tau \in \t\setminus \{\sigma_1, \dots, \sigma_k \}: \{x_\tau, y_\tau \} \bot M^{-1} e_j\}. 
$$
Suppose that  every $\tau \in \t \setminus \{\sigma_1, \dots, \sigma_k \}$
is unsuitable. Then
$$
\t\setminus \{\sigma_1, \dots, \sigma_k \} = E_1 \cup \dots \cup E_k.
$$ 
Pick $j_0$ such that $E_{j_0}$ is uncountable.
Thus, for every $\tau \in E_{j_0}$,
$$
\langle x_\tau, M^{-1}e_{j_0} \rangle = 0= \langle y_\tau , M^{-1}e_{j_0} \rangle .
$$
Let 
$$
M^{-1}e_{j_0}= \begin{bmatrix} c_1 \\ \vdots \\ c_n\end{bmatrix}.
$$
By equations \eqref{x-y-lambda},
$$
\langle x_\tau, M^{-1}e_{j_0} \rangle = \sum_{i=1}^n \frac{c_i}{1-\overline{\sigma_i}\tau} =0
$$
and
$$ 
\langle y_\tau , M^{-1}e_{j_0} \rangle  = \sum_{i=1}^n \frac{c_i \overline{\eta_i}}{1-\overline{\sigma_i}\tau} =0.
$$
Since the functions $ f_i(\lambda)=\frac{1}{1-\overline{\sigma_i}\lambda}, \; i =1, \dots, n, $ restricted to the infinite bounded set $ E_{j_0} \subset \C$ are linearly independent,
$c_i = 0$ for all $i$ and $M^{-1}e_{j_0}=0$. This is impossible. Therefore 
there exists $\tau\in \t \setminus \{\sigma_1, \dots, \sigma_k \}$ such that the equalities
$$
\langle x_\tau, M^{-1}e_j \rangle = 0= \langle y_\tau , M^{-1}e_j \rangle 
$$
do not hold  for any $j$, $1\le j \le k$. Hence there exists
$\tau\in \t \setminus \{\sigma_1, \dots, \sigma_k \}$ such that  the set $Z_\tau$ consists of at most $k$ points.
 \qed \end{proof}

Our final result concerning Problem {\rm \ref{prob}} is that the particular solution guaranteed by Proposition \ref{prop6.10} is uniquely determined by $\zeta$ and varies linear-fractionally in $\zeta$.  We suppose that  Blaschke interpolation data $(\si,\eta,\rho)$  are given, as in Definition \ref{blaschkedata}.
\begin{theorem}\label{thm6.10}
Let  the Pick matrix $M$ for  Problem {\rm \ref{prob}}
 be positive definite, and let $\tau \in \T\setminus\{\si_1,\dots,\si_k\}$ be such that the set
$$
Z_\tau=\{\zeta \in \t: u_{\zeta, \tau} \bot M^{-1} e_j \;\text{ for some} \; j, \; 1 \le j \le k\}
$$
contains at most $k$ points, where $u_{\zeta, \tau}$ is defined by equation \eqref{u-zeta-tau}.
\begin{enumerate}
\item If $\zeta \in \t \setminus Z_\tau$, then there is a unique solution $\phi_\zeta$ of Problem {\rm \ref{prob}} that satisfies  $\phi_\zeta(\tau)=\zeta$. 
\item There exist unique polynomials $a_\tau$, $b_\tau$, $c_\tau$, and $d_\tau$ of degree at most  $n$ such that
\be\label{eq6.30}
\begin{bmatrix}a_\tau(\tau)&b_\tau(\tau)\\ c_\tau(\tau)&d_\tau(\tau)\end{bmatrix} = \begin{bmatrix}1&0\\ 0&1\end{bmatrix}
\ee
and, for all $\zeta \in \t$, if $\phi$ is a solution of 
 Problem {\rm \ref{prob}} such that $\phi(\tau) = \zeta$, then
\be\label{eq6.40}
\phi(\lambda)=\frac{a_\tau(\lambda) \zeta + b_\tau(\lambda)}{c_\tau(\lambda) \zeta + d_\tau(\lambda)}
\ee
for all  $\lambda \in \d$.
\item If $\tilde a, \tilde b, \tilde c, \tilde d$ are rational functions satisfying the  equation 
\be\label{eq6.30.2}
\begin{bmatrix}\tilde a(\tau)&\tilde b(\tau)\\ \tilde c(\tau)&\tilde d(\tau)\end{bmatrix} = \begin{bmatrix}1&0\\ 0&1\end{bmatrix}
\ee
and such that for three distinct points  $\zeta$  in $\t \setminus Z_\tau$,  the equation 
\be\label{eq6.40.2}
\frac{a_\tau(\lambda) \zeta + b_\tau(\lambda)}{c_\tau(\lambda) \zeta + d_\tau(\lambda)}=\frac{\tilde a(\lambda) \zeta + \tilde b(\lambda)}{\tilde c(\lambda) \zeta + \tilde d(\lambda)}
\ee
holds for all $\lambda \in \d$, then there exists a rational function $X$ such that $\tilde a=Xa_\tau, \, \tilde b=Xb_\tau, \, \tilde c=Xc_\tau$ and $\tilde d=Xd_\tau$.
\end{enumerate}
\end{theorem}
\begin{proof}
(1) By Proposition \ref{Ztau}, there exists $\tau\in \t \setminus \{\sigma_1, \dots, \sigma_k \}$ such that  the set $Z_\tau$ consists of at most $k$ points. 
Proposition \ref{prop6.10} asserts that if $M$ is positive definite and $\zeta \in \t \setminus Z_\tau$ then there exists a solution $\phi$ to Problem {\rm \ref{prob}} with $\phi(\tau)=\zeta$.   By Proposition \ref{atmost1}, the solution (when it exists) is unique.\\

(2) Let $\zeta \in \t$ be such that there is a solution $\phi$ of Problem {\rm \ref{prob}} satisfying $\phi(\tau)=\zeta$. With the setup of the proof of Proposition \ref{prop6.10}, we have  $B_{\zeta, \tau} \ge 0$ and
\be\label{eq6.50}
\ran(B_{\zeta, \tau}) = \begin{bmatrix}-M^{-1}u_{\zeta, \tau}\\ 1\end{bmatrix}^\perp.
\ee
For $\lambda \in \d$, we define a $(n+1) \times 1$ column matrix $v_{\zeta,\lambda}$, by
$$v_{\zeta,\lambda} = \begin{bmatrix}\frac{1-\overline{\eta_1}\phi(\lambda)}{1-\overline{\sigma_1}\lambda}\\ \vdots\\ \frac{1-\overline{\eta_n}\phi(\lambda)}{1-\overline{\sigma_n}\lambda}\\ \\ \frac{1-\overline{\zeta}\phi(\lambda)}{1-\overline{\tau}\lambda}\end{bmatrix},$$
and define a $(n+2) \times (n+2)$ matrix $C_{\zeta,\lambda}$ by
$$C_{\zeta,\lambda} = \begin{bmatrix} B_{\zeta, \tau} & v_{\zeta,\lambda}\\ v_{\zeta,\lambda}^* & \frac{1-|\phi(\lambda)|^2}{1-|\lambda|^2}\end{bmatrix}.$$
As $C_{\zeta,\lambda}$ is the localization of the Pick matrix for $\phi$ to the points $\sigma_1,\ldots, \sigma_n, \tau, \lambda$, it follows that $C_{{\zeta,\lambda}} \ge 0$. Hence,  equation \eqref{eq6.50} implies that
\be\label{eq6.60}
\left\langle v_{\zeta,\lambda}\ , \begin{bmatrix}-M^{-1}u_{\zeta, \tau}\\ 1\end{bmatrix}\right\rangle = 0.
\ee
Note that 
\be\label{eq6.601}
u_{\zeta, \tau} = x_\tau - \zeta y_\tau \ \text{ and } v_{\zeta,\lambda} = \begin{bmatrix}x_\lambda\\ \frac{1}{1-\overline{\tau}\lambda}\end{bmatrix} - \phi(\lambda)\begin{bmatrix}y_\lambda\\ \frac{\overline{\zeta}}{1-\overline{\tau}\lambda}\end{bmatrix}
\ee
where $n \times 1 $ vectors $x_\lambda$ and $y_\lambda$ are defined
for $\lambda \in \d^- \setminus \{ \sigma_1, \dots, \sigma_k\}$ by the formulas \eqref{x-y-lambda}.
Hence, by  equations \eqref{eq6.60} and \eqref{eq6.601}, we have 
\begin{align}
0 &=\left\langle v_{\zeta,\lambda}\ , \begin{bmatrix}-M^{-1}u_{\zeta, \tau}\\ 1
\end{bmatrix}\right\rangle \nn\\
~ & =\left\langle \begin{bmatrix} x_\lambda\\ \frac{1}{1-\overline{\tau}\lambda}
\end{bmatrix} - \phi(\lambda) \begin{bmatrix}y_\lambda\\ 
\frac{\overline{\zeta}}{1-\overline{\tau}\lambda}\end{bmatrix}\ , 
\begin{bmatrix}-M^{-1}(x_\tau - \zeta y_\tau)\\ 1\end{bmatrix} \right\rangle \nn\\
~ & =\langle (x_\lambda - \phi(\lambda)y_\lambda), -M^{-1}(x_\tau - \zeta y_\tau)\rangle  + \frac{1- \bar{\zeta}\phi(\lambda) }{1-\overline{\tau}\lambda}.
\end{align}
Therefore
\be\label{eq6.70}
\langle x_\lambda - \phi(\lambda)y_\lambda, \zeta M^{-1}y_\tau - M^{-1}x_\tau \rangle +\frac{1}{1-\overline{\tau}\lambda} - \phi(\lambda)\frac{\overline{\zeta}}{1-\overline{\tau}\lambda}= 0.
\ee
Equation \eqref{eq6.70} for $\phi(\lambda)$ yields, after simplification, 
\be\label{eq6.80}
\phi(\lambda) = \frac{A(\lambda) \zeta + B(\lambda)}{C(\lambda)\zeta + D(\lambda)},
\ee
where
\begin{align}
A(\lambda) &= -\langle x_\lambda,M^{-1} x_\tau \rangle +\frac{1}{1-\overline{\tau}\lambda} \label{eq6.90},\\
B(\lambda) &= \langle x_\lambda,M^{-1} y_\tau \rangle \label{eq6.100},\\
C(\lambda) &= -\langle y_\lambda,M^{-1} x_\tau \rangle \label{eq6.110},
\end{align}
 and
\begin{align}
D(\lambda) &= \langle y_\lambda,M^{-1} y_\tau \rangle +\frac{1}{1-\overline{\tau}\lambda}. \label{eq6.120}
\end{align}
As the right hand sides of equations \eqref{eq6.90} - \eqref{eq6.120} depend only on the prescribed data of Problem {\rm \ref{prob}},  equation \eqref{eq6.80} implies that $\phi$ is unique as claimed.

To define 
$a_\tau,b_\tau,c_\tau,d_\tau$ with the desired properties, let
$$
\pi(\la) = (1-\overline{\tau}\lambda)\prod_{j=1}^n \frac{1-\overline{\sigma_j}\lambda}{1-\overline{\sigma_j}\tau},
$$
and set 
\be \label{abcd}
a_\tau=\pi A, \;b_\tau=\pi B,\; c_\tau=\pi C, \; \text{and} \; d_\tau=\pi D.
\ee
With these definitions,  equation \eqref{eq6.30} follows immediately from equations \eqref{eq6.90}-\eqref{eq6.120} and equation  \eqref{eq6.40} follows from  equation \eqref{eq6.80}.

(3) To prove the final assertion of Theorem \ref{thm6.10}, assume that $\tilde a, \tilde b, \tilde c, \tilde d$ are rational functions satisfying  equation \eqref{eq6.30.2} and such that for 3 distinct points $\zeta$  in $ \t \setminus Z_\tau$,  equation \eqref{eq6.40.2} holds for all $\lambda \in \d$. Cross multiplication in  equation  \eqref{eq6.40.2}
yields
$$
\tilde a c_\tau\zeta^2 + (\tilde a d_\tau + \tilde b c_\tau)\zeta +\tilde bd_\tau = a_\tau \tilde c\zeta^2 + (a_\tau \tilde d + b_\tau \tilde c)\zeta +b_\tau \tilde d
$$
for 3 distinct values of $\zeta$. Hence,
\be\label{eq6.130}
\tilde a c_\tau = a_\tau \tilde c,
\ee
\be\label{eq6.140}
\tilde ad_\tau + \tilde b c_\tau = a_\tau \tilde d + b_\tau \tilde c
\ee
and
\be\label{eq6.150}
\tilde bd_\tau = b_\tau \tilde d.
\ee
Solving  equation \eqref{eq6.130} for $ \tilde c$ and  equation \eqref{eq6.150} for $\tilde b$ and then substituting into  equation \eqref{eq6.140}, we deduce that
$$
\frac{\tilde a}{a_\tau} = \frac{\tilde d}{d_\tau}.
$$
Here, as  equation \eqref{eq6.30} guarantees that $a_\tau$ and $d_\tau$ are not identically zero, $\tilde a/a_\tau$ and $\tilde d/d_\tau$ are well defined rational functions.
Since equations \eqref{eq6.130} and  \eqref{eq6.150} imply that
$$
\tilde c= \frac{\tilde a}{a_\tau}c_\tau \text{ and } \tilde b = \frac{\tilde d}{d_\tau}b_\tau,
$$
the final assertion of the theorem follows with $X=\tilde a/a_\tau$.

To see the uniqueness of polynomials $a_\tau,b_\tau,c_\tau,d_\tau$ assume that there is a second collection $a_1,b_1,c_1,d_1$ of polynomials of degree $\le n$ such that  equations \eqref{eq6.30.2} and \eqref{eq6.40.2} hold. By what was proved in the previous paragraph, it is not  the case that both collections of polynomials are relatively prime. Otherwise, there is a third collection $a_2,b_2,c_2,d_2$ of polynomials of degree $\le n-1$ such that  equations \eqref{eq6.30} and \eqref{eq6.40} hold. This  contradicts the fact that $\deg(\phi)=n$ for all $\zeta \in \t \setminus Z_\tau$.
 \qed \end{proof}

In view of Theorem \ref{thm6.10} we can make precise what we mean by a parametrization of the solutions of a  Blaschke interpolation problem.
\begin{defin} \label{deflfparam}
Let $(\si,\eta,\rho)$ be  Blaschke interpolation data, with $n$ distinct interpolation nodes of which $k$ lie in $\t$.
Suppose that Problem {\rm \ref{prob}} is solvable. We say that 
\[
\ph=\frac{a\zeta+b}{c\zeta+d}
\]
is a {\em normalised linear fractional parametrization of the solutions of Problem  \ref{prob}} if
\begin{enumerate}
\item $a,b,c,d$ are polynomials of degree at most $n$;
\item for all but at most $k$ values of $\zeta\in\t$, the function
\be\label{lfform}
\ph(\la)=\frac{a(\la)\zeta+b(\la)}{c(\la)\zeta+d(\la)}
\ee
is a solution of Problem {\rm \ref{prob}};
\item for some point $\tau\in\t\setminus \{\si_1,\dots,\si_k\}$,
\[
\bbm a(\tau)&b(\tau) \\ c(\tau)&d(\tau) \ebm  = \bbm 1 & 0\\0 & 1 \ebm;
\]
\item  every solution $\ph$ of Problem {\rm \ref{prob}} has the form \eqref{lfform} for some $\zeta\in\t$.
\end{enumerate}
\end{defin}
\begin{remark}\label{explicit-solution} {\rm Let $(\si,\eta,\rho)$ be  Blaschke interpolation data, with $n$ distinct interpolation nodes of which $k$ lie in $\t$. Suppose the Pick matrix $M$ of this problem is positive definite.
The above proof of Theorem \ref{thm6.10} gives an {\em explicit}  linear fractional parametrization of the solutions of Problem \ref{prob}.  As in Theorem \ref{thm6.10} choose 
$\tau \in \T\setminus\{\si_1,\dots,\si_k\}$ such that the set
$Z_\tau$ contains at most $k$ points. A normalised linear fractional parametrization of the solution set of Problem {\rm \ref{prob}} is
\[
\ph=\frac{a_\tau\zeta+b_\tau}{c_\tau \zeta+d_\tau},
\]
 where 
the polynomials
$ a_\tau$, $b_\tau$, $ c_\tau$ and $ d_\tau$ are defined by equations \eqref{abcd}.  Note that different choices of $\tau$  will yield different normalised parametrizations.
}
\end{remark}
In the terminology of Definition \ref{deflfparam}, Theorem \ref{thm6.10} tells us the following.

\begin{corollary}\label{norm-lin-param} Let $(\si,\eta,\rho)$ be  Blaschke interpolation data, with $n$ distinct interpolation nodes.
Suppose the Pick matrix $M$ of this problem is positive definite. There exists a normalized linear fractional parametrization 
\[
\ph=\frac{a\zeta+b}{c\zeta+d}
\]
of the solutions of Problem {\rm \ref{prob}}. 
Moreover
\begin{enumerate}
\item  at least one of the polynomials
$a,b,c,d$ has degree  $n$;
\item the polynomials
$a,b,c,d$  have no common zero in $\C$;
\item $|c| \le |d|$ on $\d^-$.
\end{enumerate}
\end{corollary}
\begin{proof} 
As in Theorem \ref{thm6.10} choose 
$\tau \in \T\setminus\{\si_1,\dots,\si_k\}$ such that the set
$Z_\tau$ contains at most $k$ points. Let the polynomials
$a = a_\tau$, $b= b_\tau$, $c= c_\tau$ and $d= d_\tau$ be defined by equations \eqref{abcd}.  
Theorem \ref{thm6.10} shows that $(a,b,c,d)$ has the properties (1), (2) and (3) of Definition \ref{deflfparam}. 
Let  $\phi$ be a solution of 
 Problem {\rm \ref{prob}} and let $\zeta =\phi(\tau)$. By  Theorem \ref{thm6.10}(2), $\phi$ is given by equation \eqref{lfform}. Hence property (4) of Definition \ref{deflfparam} holds.

Moreover (1) if all of $a,b,c,d$ have degree strictly less than $n$ then $\ph=\frac{a\zeta+b}{c\zeta+d}$ is a rational function of degree strictly less than $n$, and so is not a solution of   Problem \ref{prob}.

(2) Suppose $\alpha \in \C$ is a common zero of the polynomials $a,b,c,d$. On cancelling the common factor $\la-\al$ above and below in equations \eqref{lfform} and multiplying numerator and denominator by a suitable nonzero scalar we obtain a different normalized parametrization of solutions of Problem \ref{prob}, with the same $\tau$, contrary to the uniqueness statement in Theorem \ref{thm6.10}(2).  Hence $a,b,c$ and $d$ have no common zero in $\c$.

(3) 
By the normalization property in  Definition \ref{deflfparam}(3),
\[
(ad-bc)(\la) \to 1 \quad \mbox{ as }\la\to \tau.
\]
Hence $ad-bc$ is a polynomial of degree at most $2n$ and is not identically zero.  Therefore
\[
Y \df \{\la\in\d : (ad-bc)(\la)=0\}
\]
contains at most $2n$ points.

We claim that the real-valued function $f(\la)=|d(\la)|-|c(\la)|$ has no zeros in $\d\setminus Y$.  For suppose that $\la_0$ is a zero of $f$.  Then there exists $\zeta_0\in\t$ such that $c(\la_0)\zeta_0+d(\la_0)=0$.  Since $|\ph|=\left|\frac{a\zeta+b}{c\zeta+d}\right|\leq 1$ on $\d$ for almost all $\zeta\in\t$, it follows that also $a(\la_0)\zeta_0+b(\la_0)=0$, and therefore $(ad-bc)(\la_0)=0$, that is, $\la_0\in Y$.

Since $c(\tau)=0$ and $d(\tau)=1$, the continuous function $f$ is strictly positive on a neighbourhood of $\tau$ in $\d$.  Suppose that $f(\la_1) <0$ for some $\la_1\in\d$.  Then $f<0$ on an open set, and hence there are infinitely many points in $\d$ at which $f=0$, a contradiction.  Hence $f\geq 0$ on $\d$.
\qed
\end{proof}


\section{Prescribing the nodes and values} \label{sec7}

In this section we shall show how to construct rational $\gaminn$ functions with prescribed royal nodes and values. Our answer will be in terms of the solution to Problem {\rm \ref{prob}} as described in Proposition \ref{prop6.10} and Theorem \ref{thm6.10}.  First we require a notion of multiplicity for royal nodes.

\begin{definition}\label{def3.30}
Let $h$ be a rational $\gaminn$ function with royal polynomial $R$.  If  $\sigma$ is a zero of $R$ of order $\ell$, we define the \emph{multiplicity $\#\si$} of $\si$ (as a royal node of $h$) by 
\[
 \#\si \quad = \quad \twopartdef{ \ell} {\mbox{ if }\sigma \in \d}{\half \ell}{ \mbox{ if }\sigma \in \t.}
\]
The {\em type} of $h$ is the ordered pair $(n,k)$ where $n$ is the sum of the multiplicities of the royal nodes of $h$ that lie in $\d^-$ and  $k$ is the sum of the multiplicities of the royal nodes of $h$ that lie in $\t$. We denote by $\rnk$ the collection of rational $\gaminn$ functions $h$ of type $(n,k)$.
\end{definition}
By \cite[Theorem 3.8]{ALY14}, if $h=(s,p)$ belongs to $\rnk$ then $\deg(h)= n$ and $p$ is a Blaschke product of degree $n$.

The following example of rational $\gaminn$ functions from $
\mathcal{R}^{n,k}$ for even $n \ge 2$ can be found in \cite[Proposition 12.1]{ALY12}.
\begin{example}{\rm
 For all $\nu \ge 0$ and $0<r<1$,  the function 
\begin{equation}\label{h_nu}
h_{\nu}(\la) =\left( 2(1-r) \frac{\la^{\nu+1}}{1+ r\la^{2\nu+1}},
\frac{\la(\la^{2\nu+1}+r)}{1+ r\la^{2\nu+1}} \right), \;\la \in \D,
\end{equation}
belongs to $
\mathcal{R}^{2 \nu +2,2 \nu +1}$.
The royal nodes of $h_\nu$ that lie in $\t$, being the points at which $|s|=2$, are the $(2\nu +1)$th roots of $-1$, that is, 
\[
\omega_j = e^{\ii \pi (2j+1)/(2\nu+1)}, \; j= 0, \dots, 2 \nu.
\]
They are all of multiplicity $1$.
Note that there is a simple royal node at $0$.
}
\end{example}

In this section we are concerned only with rational $\Ga$-inner functions whose royal nodes all have multiplicity $1$.

The following elementary calculation will be useful.
\begin{lemma}\label{s-c-d} 
Let $a,b,c,d,s_0,p_0\in\c$ and suppose that $|\p0| =1$,  
$\s0 = \overline{\s0}\p0$, \, $s_0c\neq 2d$ and $|s_0|<2 $.  Let
\be\label{eq8.021}
s=2\frac{2p_0c-s_0d}{s_0c-2d}.
\ee
Then
\be\label{eq8.022}
|s| \le 2 \;  \iff \; |c| \le |d|.
\ee
\end{lemma}
\begin{proof}  
\begin{align}
\label{s-c-d1}
|s| \le 2 &\iff |2p_0c-s_0d|^2 \le |s_0c-2d|^2\\
~ &\iff  4|c|^2 -2 {\rm Re} (2 p_0 c \bar{s_0} \overline{d}) + |s_0|^2 |d|^2  \nn\\
~ & \qquad \le |s_0|^2 |c|^2-2 {\rm Re} (2 s_0 c \overline{d}) +4|d|^2  \nn
\\
~ &\iff  (4- |s_0|^2)(|c|^2 -|d|^2) \le 4 {\rm Re} (s_0 c \overline{d} - s_0 c \overline{d}) \nn \\
~ &\iff  |c|^2 - |d|^2 \leq 0 \nn\\
\label{s-c-d2}
 &\iff |c| \le |d|.
\end{align}
\qed
\end{proof}

The next result provides a necessary condition for the existence of a rational $\Ga$-inner function with prescribed royal interpolation data.
\begin{thm}\label{thm7.10}
Let $h=(s,p)$ be a rational  $\Gamma$-inner function of type $(n,k)$ having distinct royal nodes $\si_1,\dots,\si_n$ and corresponding royal values $\eta_1,\dots,\eta_n$, where $\si_1,\dots,\si_k \in\t$.  Let $\rho_j = \half Ap(\si_j)$ for $j=1,\dots,k$.
\begin{enumerate}
\item There exists a rational inner function $\ph$  that solves Problem {\rm \ref{prob}}, that is, such that
 $\deg(\ph) = n$,
\be\label{17.1}
\ph(\si_j)= \eta_j \quad \mbox{ for } \quad j=1,\dots,n
\ee
and
\be\label{17.2}
A\ph(\si_j)= \rho_j \quad \mbox{ for } \quad j=1,\dots,k.
\ee
Any such function $\ph$ is expressible in the form
$\ph = \Phi_\omega\circ h$ for some $\omega \in \t$.
\item There exist polynomials $a,b,c,d$ of degree at most $n$ such that a normalized  parametrization of the solutions of
Problem {\rm \ref{prob}} is
\[
\ph= \frac{a\zeta+b}{c\zeta+d}, \quad \zeta \in\t.
\]
\item  For any polynomials $a,b,c,d$ as in {\rm (2)}, there exist $s_0,p_0 \in\c$ such that 
\begin{align}
\label{eq7.170}
|\p0| &=1,  \\
\label{eq7.180}
\s0 &= \overline{\s0}\p0, \\
\label{eq7.190}
|s_0|&<2 \qquad \\
\label{eq7.210}
|c| &\le |d|,
\end{align}
\be\label{eq7.200}
 \s0 a - 2b + 2\p0 c - \s0 d = 0
\ee
and
\be\label{eq7.220}
(2p_0c- s_0d)^2 \neq (-2p_0a+s_0b)(s_0c -2d).
\ee
\end{enumerate}
Moreover 
\begin{align}
\label{eq7.231}
s&=2\frac{2\p0 c - \s0 d}{\s0 c - 2d},\\
\label{eq7.232}
p&=\frac{-2\p0 a + \s0 b}{\s0 c - 2d}
\end{align}
\end{thm}
\begin{proof}
(1)   For $\omega \in \t$ consider the rational function 
\be\label{eq7.10}
\psi_\omega=\Phi_\omega\circ h=\frac{2\omega p-s}{2-\omega s}.
\ee
Then, if $\omega \ne -\bar\eta_1,-\bar\eta_2,\ldots,-\bar\eta_k$,
\be\label{eq7.20}
\psi_\omega(\sigma_j)=\frac{2\omega\eta_j^2+2\eta_j}{2+\omega2\eta_j}=\eta_j \quad \mbox{ for } j=1,2,\ldots,n.
\ee

  We claim that, for $\omega\in\t\setminus\{-\bar\eta_1. \dots,\bar\eta_k\}$,
the function $\ph=\psi_\omega$ is a solution of Problem \ref{prob}.
By \cite[Proposition 3.2]{ALY12}, for any $\omega\in\T$ and any point $(s(\lambda),p(\lambda))\in\Gamma$,
\[
|\Phi_\omega(s(\lambda),p(\lambda))| = 1 \quad\mbox{ if and only if }\quad\omega(s(\lambda)-\bar s(\lambda) p(\lambda))= 1-|p(\lambda)|^2.
\]
Thus it is easy to see that $\ph$ is inner.
The equation \eqref{eq7.20} shows that $\psi_\omega$ takes the required values
at  $\si_1,\dots,\si_n$.  
By Proposition \ref{phaser_h}(iv),
\be\label{eq7.30}
A\psi_\omega(\sigma_j) =\half Ap(\si_j)= \rho_j \quad \mbox{ for } j=1,2,\ldots,k.
\ee
It is also true that $\deg(\psi_\omega)=n$ for $\omega\neq -\bar\eta_1. \dots, -\bar\eta_k$.  By \cite[Proposition 2.2]{ALY14}, for a rational  $\Gamma$-inner function $h =(s,p)$ such that $\deg(p)=n$ and if $D$ is the denominator when $p$ is written in its lowest terms then $s$ can also be written with denominator  $D$.  It follows that 
\be\label{degrees}
\deg(\psi_\omega)= \deg(p)-\#\{\mbox{cancellations between } 2\omega p-s \mbox{ and }2-\omega s\}.
\ee
By \cite[Theorem 7.12]{ALY12}, such cancellations can occur only at the royal nodes $\si_j \in \t, \, j=1,\dots,k$, and then only when $\omega=\half \overline{s(\si_j)}=-\bar\eta_j$.  Hence there are {\em no} cancellations in equation \eqref{degrees}, and so $\deg(\psi_\omega)=n$.
We have shown that, if $\omega \ne -\bar\eta_1,-\bar\eta_2,\ldots,-\bar\eta_k$, then $\ph=\psi_\omega$ is a solution of Problem {\rm \ref{prob}}.

\noindent (2)   Since Problem \ref{prob} is solvable, its Pick matrix is positive definite and so Theorem \ref{thm6.10} tells us that there exist polynomials $a,b,c,d$ of degree at most $n$ which parametrise the solutions of Problem \ref{prob}.   Let us choose a particular such 4-tuple of polynomials, as described in  Theorem \ref{thm6.10}.    By Proposition \ref{Ztau}, there exists $\tau\in \t \setminus \{\sigma_1, \dots, \sigma_k \}$ such that  the set $Z_\tau$ (defined  in equation \eqref{defZtau}) consists of at most $k$ points.
Fix such a $\tau \in \t$; then there exist unique polynomials $a_\tau,b_\tau,c_\tau,d_\tau$ of degree at most $n$ such that 
\be\label{normaliseattau}
\bbm a_\tau(\tau) & b_\tau(\tau) \\ c_\tau(\tau) & d_\tau(\tau) \ebm = \bbm 1&0 \\ 0&1 \ebm
\ee
and, for all $\zeta\in \t\setminus Z_\tau$, the function
\be\label{lftau}
\ph= \frac {a_\tau\zeta+b_\tau}{c_\tau\zeta+d_\tau}
\ee
is the unique solution of Problem \ref{prob} that satisfies $\ph(\tau)=\zeta$.  Moreover, the general 4-tuple of polynomials that parametrises the solutions of Problem \ref{prob} is expressible in the form
\be\label{someX} 
(a,b,c,d)=(Xa_\tau, Xb_\tau, Xc_\tau, Xd_\tau)
\ee
 for some rational function $X$.

Let $\s0=s(\tau)$, $\p0=p(\tau)$.   Since $h$ is $\Ga$-inner, equations \eqref{eq7.170} and \eqref{eq7.180} hold by virtue of Proposition \ref{prop2.10}.  Since $\tau$ is chosen not to be a royal node of $h$, the inequation \eqref{eq7.190} also holds.  Moreover  $|s_0|< 2$, since, for any point $(z_1,z_2)$ in 
the distinguished boundary $\M$ of $\gam$,
 we have $|z_1|=2$ if and only if $z_1^2=4z_2$ -- see \cite[Proposition 3.2(3)]{ALY12}.  It remains to prove equations \eqref{eq7.200} and \eqref{eq7.210}.

\begin{lem}\label{fact7.20}
Let
$$
 Z_\tau^\sim  \stackrel{ \rm def}{=} \left\{ \frac{-2\overline{\eta_1}\p0 -\s0}{2 + \overline{\eta_1} \s0}, \frac{-2\overline{\eta_2}\p0 -\s0}{2 + \overline{\eta_2} \s0},\ldots, \frac{-2\overline{\eta_k}\p0 -\s0}{2 + \overline{\eta_k} \s0}\right\}.
$$
If $\zeta \in \t\setminus Z_\tau^\sim$
then the function
\be\label{eq7.50}
\phi= \frac{(4p-\s0 s)\zeta + 2\s0 p -2\p0 s}{(2\s0 -2s)\zeta + 4\p0 - \s0 s}
\ee
is a solution of Problem {\rm \ref{prob}} and satisfies $\phi(\tau) = \zeta$.
\end{lem}
\begin{proof}
Observe that, by equation \eqref{eq7.10}, for any $\omega\in\t$,
\[
\psi_\omega(\tau) = \frac{2\omega p_0-s_0}{2-\omega s_0},
\]
which is well defined since $|s_0|<2$.  We have, for $\zeta\in\t$,
\begin{align*}
\psi_\omega(\tau)=\zeta \quad &\Leftrightarrow\quad \frac{2\omega \p0 - \s0}{2-\omega \s0}=\zeta\\
	&\Leftrightarrow \quad \omega = \frac{2\zeta +\s0}{2\p0 +\zeta \s0}.
\end{align*}
Hence, as long as 
\be\label{1621}
\frac{2\zeta+s_0}{2p_0+\zeta s_0} \neq -\bar\eta_1,\dots,-\bar\eta_k,
\ee
the function
\be\label{gotphi}
\phi= \psi_\omega = \psi_{\frac{2\zeta +\s0}{2\p0 +\zeta \s0}}
\ee
is a solution of Problem \ref{prob} which satisfies in addition $\ph(\tau)=\zeta$.   Condition \eqref{1621} can equally be written 
\[
\zeta \neq -\frac{2\bar\eta_j p_0+s_0}{2+\bar\eta_js_0}, \quad \mbox{ for } j=1,2,\dots,k
\]
or equivalently $\zeta \notin Z_\tau^\sim$.

 On computing $\ph$ from equations \eqref{gotphi} and \eqref{eq7.10} we find that $\ph$ is indeed given by equation \eqref{eq7.50}; this establishes the Lemma.
\qed \end{proof}
We conclude the proof of Theorem \ref{thm7.10} (2).  For $\zeta\in\t\setminus(Z_\tau\cup Z_\tau^\sim)$ we have two expressions for the unique solution of Problem \ref{prob} for which $\ph(\tau)=\zeta$, to wit equations \eqref{lftau} (with the normalising condition \eqref{normaliseattau}) and \eqref{eq7.50}.
Note that
\begin{align*}
\begin{bmatrix}4p(\tau)-\s0 s(\tau) & 2\s0 p(\tau) -2\p0 s(\tau)\\
2\s0 -2s(\tau) & 4\p0 - \s0 s(\tau)\end{bmatrix} &=
\begin{bmatrix}4\p0-\s0 \s0 & 2\s0 \p0 -2\p0 \s0\\
2\s0 -2\s0 & 4\p0 - \s0 \s0\end{bmatrix}\\
&=(4\p0 - \s0^2)\begin{bmatrix}1 & 0\\0 & 1\end{bmatrix}.
\end{align*}
Since the set $Z_\tau \cup Z_\tau^\sim$ is finite, the linear fractional transformations in equations \eqref{lftau} and \eqref{eq7.50} are equal at infinitely many points, hence coincide.   On taking account of the normalising condition we obtain
\[
\bbm a_\tau & b_\tau \\ c_\tau & d_\tau \ebm = \frac{1}{4p_0-s_0^2} \bbm 4p-s_0s & 2s_0p-2p_0s \\ 2s_0-2s & 4p_0-s_0s \ebm.
\]

Suppose that $a$, $b$, $c$, and $d$ are polynomials that parametrise the solutions of Problem \ref{prob}, as in Theorem \ref{thm7.10} (2).  By the observation \eqref{someX}, there exists a rational function $X$ such that
\begin{align}
\label{eq7.60}
Xa&= 4p-\s0 s,\\
\label{eq7.70}
Xb&=2\s0 p -2\p0 s, \\
\label{eq7.80}
Xc&=2\s0 -2s, \text{ and}\\
\label{eq7.90}
Xd&=4\p0 - \s0 s.
\end{align}
Thus
\begin{align*}
X(s_0a-2b&+2p_0c-s_0d)=\\
	& s_0(4p-s_0s)-2(2s_0p-2p_0s)+2p_0(2s_0-2s)-s_0(4p_0-s_0s)
\end{align*}
which is zero.  Hence equation \eqref{eq7.200} holds.

Let us find connections between $s,p$ and the polynomials $a,b,c,d$.
Solving equations \eqref{eq7.80} and \eqref{eq7.90} for $s$ and $X$ we find that
\be\label{eq8.01}
X =\frac{2 s_0^2 -8p_0}{s_0c-2d}
\ee
and 
\be\label{eq8.02}
s=2\frac{2p_0c-s_0d}{s_0c-2d}.
\ee
Eliminating $s$ from equations \eqref{eq7.60} and \eqref{eq7.70} we deduce that
\be\label{eq7.120}
(8\p0 - 2\s0^2)p = X(2\p0 a - \s0b),
\ee
which implies via equation \eqref{eq8.01} that
\be\label{eq8.03}
p=\frac{-2p_0a+s_0b}{s_0c-2d}.
\ee

Since $h=(s,p)$ is a rational  $\Gamma$-inner function, $|s| \le 2$ on $\d$ and, by Lemma \ref{s-c-d},  equation \eqref{eq7.210} holds. 
Since, by assumption,   $h(\d) \not\subset \royal$, we have $s^2 \neq 4p$ on $\d^-$ and 
inequation \eqref{eq7.220} holds.
The proof of Theorem \ref{thm7.10} is complete.
\qed
\end{proof}
\begin{remark} \rm
The above proof shows that, if $\omega\neq -\bar\eta_1,\dots,-\bar\eta_k$, then $\Phi_\omega\circ h$ is a solution of the corresponding Blaschke interpolation problem.  What if $\omega=-\bar\eta_j$ for some $j\in\{1,\dots,k\}$?  Then the rational function $\Phi_\omega\circ h$ has a removable singularity at $\si_j$.  After cancellation, it is still true (by Proposition \ref{phaser_h})
that $\Phi_\omega\circ h(\si_j)=\eta_j$, but we cannot assert that $A(\Phi_\omega\circ h)(\si_j)=2\rho_j$.  In any case $\Phi_\omega\circ h$ has degree $n-1$, and so is not a solution of Problem \ref{prob}.
\end{remark}

There is a converse to Theorem \ref{thm7.10}.  To prove it we need the following purely algebraic observation, which is proved by a routine calculation.
\begin{prop}\label{purealg}
Let $a,b,c,d$ be polynomials in the indeterminate $\la$ and suppose that $s_0,p_0\in\c$ satisfy $s_0^2\neq 4p_0, \, s_0c\neq 2d$ and
\[
s_0a-2b+2p_0c-s_0d=0.
\]
Let rational functions $s,p$ be defined by
\be\label{8.1}
s=2\frac{2p_0c-s_0d}{s_0c-2d}, \qquad p=\frac{-2p_0a+s_0b}{s_0c-2d}
\ee
and let 
\be\label{defzetao}
\zeta(\omega)= \frac{2\omega p_0-s_0}{2-\omega s_0}.
\ee
Then, as rational functions in $(\omega,\la)$,
\[
\frac{2\omega p(\la)-s(\la)}{2-\omega s(\la)} = \frac{a(\la)\zeta(\omega)+b(\la)}{c(\la)\zeta(\omega)+d(\la)}.
\]
\end{prop}

This algebraic relation has implications for rational maps from $\d$ to $\Ga$.
\begin{prop}\label{p8(2)}
Let $a,b,c,d$ be polynomials having no common zero in $\d^-$ and satisfying $|c|\leq|d|$ on $\d$.  Suppose that $s_0,p_0\in\c$ satisfy $s_0c\neq 2d$ and
\[
s_0a-2b+2p_0c-s_0d=0.
\]
Suppose in addition that $|p_0|=1, \, |s_0| < 2$ and $s_0=\bar s_0 p_0$.
Let rational functions $s,p$ be defined by equations \eqref{8.1} and let
\be\label{leq11}
\psi_\zeta(\la)  = \frac{a(\la)\zeta+b(\la)}{c(\la)\zeta+d(\la)}.
\ee

{\rm (i)} If, for all but finitely many values of $\la\in\d$, 
\be\label{leq1}
\left|\psi_\zeta(\la)\right| \leq 1
\ee
for all but finitely many $\zeta\in\t$, then  $s_0c-2d$ has no zeros in $\d^-$ and $(s,p)$ is a holomorphic map from $\d$ to $\Ga$.

{\rm (ii)} If, for all but finitely many $\zeta\in\t$, the function $\psi_\zeta $  is inner, then $h=(s,p)$ is a rational $\Ga$-inner function. 
\end{prop}
\begin{proof}(i)
Notice first that the hypotheses on $s_0$ and $p_0$ imply that $\zeta(\cdot)$ (defined by equation \eqref{defzetao}) is an automorphism of $\d$ and so defines a bijective self-map of $\t$.  

By hypothesis there is a finite subset $E$ of $\d$ such that, for all $\la\in\d\setminus E$, there is a finite subset $F_\la$ of $\t$ such that the inequality \eqref{leq1} holds for all $\zeta\in \t\setminus F_\la$.  

We claim that the denominator $s_0c-2d$ of $s$ and $p$ in equations \eqref{8.1} has no zeros in $\d^-$.  For suppose that $\al$ is such a zero.
Since $|c|\leq |d|$ on $\d^-$ and $|s_0| < 2$,
\begin{align*}
0=|s_0c-2d| &\geq 2|d|-|s_0c| \\
	&\geq  (2-|s_0|)|d|
\end{align*}
at $\al$, and hence $d(\al)=0$, and consequently $c(\al)=0$. 

  Choose a sequence $\al_j$ in $\d\setminus E$ such that $\al_j \to \al$.  For each $j$, for $\zeta\in\t\setminus F(\la_j)$ we have $|\psi_\zeta|\leq 1$ on $\d\setminus E$.  Hence, for all but countably many $\zeta\in\t$ (that is, for $\zeta\in\t\setminus\cup_j F(\la_j)$)
\[
\left| \frac{a(\al_j)\zeta+b(\al_j)}{c(\al_j)\zeta+d(\al_j)}\right| \leq 1.
\]
Since $c(\al_j)\zeta+d(\al_j) \to 0$ uniformly almost everywhere for $\zeta\in\t$ as $j\to \infty$, the same holds for
 $a(\al_j)\zeta+b(\al_j)$. Hence $a(\al_j)\to 0$ and $b(\al_j)\to 0$.  Thus $a(\al)=b(\al)=0$.  Hence $a,b,c,d$ all vanish at $\al$, contrary to hypothesis.  It follows that $s_0c-2d$ has no zeros in $\d^-$.

Thus $s$ and $p$ are rational functions having no poles in $\d^-$.

Consider $\la\in \d\setminus E$.
By Proposition \ref{purealg},
\be\label{ratident}
\Phi_\omega(s(\la),p(\la)) = \frac{a(\la)\zeta(\omega)+b(\la)}{c(\la)\zeta(\omega)+ d(\la)}
\ee
whenever both sides are defined, that is, for all $\omega\in\t\setminus\Omega_\la$ where
\[
\Omega_\la=\{ \omega\in\t: \omega s(\la)= 2\quad \mbox{ or }\quad c(\la) \zeta(\omega)= -d(\la)\}.
\]
$\Omega_\la$ contains at most two points.
On combining the relations \eqref{leq11},  \eqref{leq1} and \eqref{ratident} we deduce that 
\[
\left|\Phi_\omega(s(\la),p(\la))\right| \leq 1
\]
for all $\omega\in\t$ such that $\omega \notin \Omega_\la\cup \zeta^{-1}(F_\la)$, hence for all but finitely many
$\omega \in\t$. By Lemma \ref{s-c-d}, $|s(\la)| \le 2$.  It follows from \cite[Proposition 3.2(2)]{ALY12} that $(s(\la),p(\la))\in\Ga$.  Since this is true for all but finitely many $\la\in\d$ and $s,p$ are rational functions without poles in $\d^-$,  $(s,p)$ maps the whole of $\d^-$ into $\Ga$. 

(ii)  Suppose that, for some finite subset $F$ of $\t$, the function $\psi_\zeta$ is inner for all $\zeta\in\t\setminus F$. 
By Part (i), $(s,p)$ maps $\d$ into $\Ga$ and therefore extends to a continuous map of $\d^-$ into $\Ga$.   Consider $\la\in\t$.  
 By Proposition \ref{purealg} and equation \eqref{leq11},
\be\label{ratident-2}
\Phi_\omega(s(\la),p(\la))= \psi_{\zeta(\omega)}(\la)
\ee
whenever both sides are defined, that is, for all $\omega\in\t\setminus\Omega_\la$ where
\[
\Omega_\la=\{ \omega\in\t: \omega s(\la)= 2\quad \mbox{ or }\quad c(\la) \zeta(\omega)= -d(\la)\}.
\]
$\Omega_\la$ contains at most two points. 
For $\omega\in\t\setminus \zeta^{-1}(F)$ the function $\psi_{\zeta(\omega)}$ is inner.  Hence, for $\omega\in\t\setminus (\zeta^{-1}(F)\cup \Omega_\la)$,
\be\label{mod1}
|\Phi_\omega(s(\la),p(\la))| = |\psi_{\zeta(\omega)}(\la)| =1.
\ee
 \cite[Proposition 3.2]{ALY12} asserts that, for any $\omega\in\T$ and any point $(s_1,p_1)\in\Gamma$,
\[
|\Phi_\omega(s_1,p_1)| = 1 \quad\mbox{ if and only if }\quad\omega(s_1-\bar s_1 p_1)= 1-|p_1|^2.
\]
Hence, if  $|\Phi_\omega(s_1,p_1)| = 1$ for two distinct $\omega\in\t$, then 
$|p_1|=1$ and  $s_1= \bar{s_1} p_1$, which is to say that $(s_1,p_1)$ is in the distinguished boundary $\M$ of $\gam$.
Therefore, since equation \eqref{mod1} holds for many $\omega\in\t$,  $(s(\la), p(\la))\in\M$. 
Thus $h=(s,p)$ is a rational $\Ga$-inner function. 
\qed
\end{proof}

The following result gives the promised explicit construction of a solution of the royal $\Ga$-interpolation problem in terms of a normalized parametrization of solutions of the corresponding Blaschke interpolation problem.
\begin{theorem}\label{thm7.20}
Let   $(\si,\eta,\rho)$ be  Blaschke interpolation data with $n$ distinct interpolation nodes of which $k$ lie in $\t$, as in Definition {\rm \ref{blaschkedata}}.   Suppose that 
 Problem {\rm \ref{prob}} with these data is solvable and 
the solutions $\ph$ of Problem {\rm \ref{prob}}  have normalized parametrization
\[
\ph= \frac{a\zeta+b}{c\zeta+d}.
\]
Suppose that there exist scalars $\s0$ and $\p0$ such that 
\be\label{s0p0bGa}
|\p0| =1, \;\s0 = \overline{\s0}\p0,\; |s_0|<2
\ee
 and
\be\label{eq1.801}
 \s0 a - 2b + 2\p0 c - \s0 d = 0.
\ee
Then there exists a rational $\Ga$-inner function $h=(s,p)$ such that

{\rm (i)}  
$h \in \rnk$,

{\rm (ii)} $h(\sigma_j) =(-2\eta_j, \eta^2_j)$ for  $j=1,2,\ldots,n$,

{\rm (iii)} $  Ap(\sigma_j)=2\rho_j$ for $j=1,2,\ldots,k$.  

{\rm (iv)}  for all  but finitely many $\omega \in \t$, the function  $ \Phi_\omega\circ h$  is a solution of Problem {\rm \ref{prob}}.

An explicit function $h=(s,p)$ satisfying conditions {\rm (i)-(iv)} is given by
\begin{align}
\label{eq7.110_2}
s&=2\frac{2\p0 c - \s0 d}{\s0 c - 2d},\\
\label{eq7.130_2}
p&=\frac{-2\p0 a + \s0 b}{\s0 c - 2d}.
\end{align}
\end{theorem}
\begin{proof} By Corollary \ref{norm-lin-param} (3),
 $|c| \le |d|$ on $ \d^- $. Hence $\left|\frac{d(\lambda)}{c(\lambda)} \right| \ge 1$ for $\lambda \in \d^-$.
By assumption $|s_0/2|<1$, and therefore $s_0c\neq 2d$ on $ \d^- $. By Proposition \ref{p8(2)},
 $h$ is a rational $\Ga$-inner function.  Since $a,b,c,d$ are polynomials of degree at most $n$, the rational function $h$ has degree at most $n$. Recall that the degree of $h$ coincides with the degree of $p$.

By Definition \ref{deflfparam} of a  normalised linear fractional parametrization of the solutions of Problem \ref{prob}, for some point $\tau\in\t\setminus \{\si_1,\dots,\si_k\}$,
\[
\bbm a(\tau)&b(\tau) \\ c(\tau)&d(\tau) \ebm  = \bbm 1 & 0\\0 & 1 \ebm.
\]
Thus it is easy to see that
\begin{align}
\label{eq7.110}
s(\tau)&=2\frac{2\p0 c(\tau) - \s0 d(\tau)}{\s0 c(\tau) - 2d(\tau)} =s_0,\\
\label{eq7.130}
p(\tau)&=\frac{-2\p0 a(\tau) + \s0 b(\tau)}{\s0 c(\tau) - 2d(\tau)}= p_0.
\end{align}
By assumption, $|\p0| =1$ and $|s_0|<2$, and hence $s(\tau)^2 \neq 4p(\tau)$.
Therefore $h(\d^-)$ is not in the royal variety $ \mathcal{R}$.

Let us show that $h$ satisfies the interpolation conditions
\be\label{hinterp}
h(\si_j)=(-2\eta_j, \eta_j^2)
\ee
for $j=1,\dots, n$, which is to say that $\si_j$ is a royal node of $h$ with corresponding royal value $\eta_j$. 
By hypothesis, there is a finite set $F\subset\t$ such that, for all $\zeta\in\t\setminus F$, the function
\[
\ph(\la)=\psi_{\zeta}(\la)\df  \frac{a(\la)\zeta+b(\la)}{c(\la)\zeta+d(\la)}
\]
is a solution of Problem \ref{prob}, and so
\be\label{psizo}
\psi_{\zeta}(\si_j) = \eta_j \quad \mbox{ for }j=1,\dots,n
\ee
and
\be\label{Apsizo}
A\psi_{\zeta}(\si_j) = \rho_j \quad \mbox{ for }j=1,\dots,k
\ee
for all $\zeta\in\t\setminus F$.
By Proposition \ref{purealg}
\be\label{lasteq}
 \psi_{\zeta(\omega)}(\la)=  \frac{a(\la)\zeta(\omega)+b(\la)}{c(\la)\zeta(\omega)+d(\la)} 
= \frac{2\omega p(\la)-s(\la)}{2-\omega s(\la)} =\Phi_\omega\circ h(\la)
\ee
as rational functions in $(\omega,\la)$, where $\zeta(\omega)= \frac{2\omega p_0-s_0}{2-\omega s_0}$.  
Hence, for $\omega\in\t\setminus \zeta^{-1}(F)$, $\Phi_\omega\circ h$ is a solution of Problem \ref{prob}; this proves statement (iv). 

For any $\la\in\d^-$ equation \eqref{lasteq} holds whenever both denominators are nonzero, hence for all but at most two values of $\omega\in\t$.  On combining equations \eqref{psizo} and \eqref{lasteq} (with $\la=\si_j$) we infer that, for $j=1,\dots,n$ and for all but finitely many $\omega\in\t$, 
\[
\frac{2\omega p(\si_j)-s(\si_j)}{2-\omega s(\si_j)} =\psi_{\zeta(\omega)}(\si_j)= \eta_j.
\]
Therefore, for almost all $\omega \in \t$, 
\[
2\omega p(\si_j)-s(\si_j)= \eta_j(2-\omega s(\si_j)).
\]
It follows that $ s(\si_j)= -2 \eta_j$ and $p(\si_j)= \eta_j^2$, $ j =1,2, \dots,n$, and so the interpolation conditions \eqref{hinterp} hold.

We have already observed that  $\deg(h)\le n$ and  that $h(\d)$ is not in  $\mathcal{R}$. Thus \cite[Theorem 3.8]{ALY14} tells us that, in this case, the number of royal nodes of $h$ is equal to the degree of $h$. Therefore $h$ has at most $n$ royal nodes.   Since
the points $\si_j$, $j =1,2, \dots, n$ {\em are} royal nodes, they comprise {\em all} the royal nodes of $h$ and $\deg(h)= n$.  Precisely $k$ of the $\si_j$ lie in $\t$, and so $h$ has exactly $k$ royal nodes in $\t$.  Thus $h\in\rnk$ and statement (i) holds.

Next we show that $Ap(\si_j)=2\rho_j$.  Fix $j\in\{1,\dots,k\}$.
By Proposition \ref{phaser_h}(iv),
 for  $\omega\in\t, \, \omega\neq -\bar\eta_j$  (and so
$2-\omega s(\si_j)= 2(1+\omega \eta_j) \neq 0$),
\be\label{first}
A(\Phi_\omega\circ h)(\si_j) = \half Ap(\si_j).
\ee
There is also a set $\Omega_j$ containing at most one $\omega_j\in\t$ such that $c(\si_j)\zeta(\omega)+d(\si_j) =0$ for $\omega  \in \Omega_j$.  Hence if $\omega\in\t\setminus  (\{-\bar \eta_j\} \cup  \Omega_j)$, it follows from equation \eqref{lasteq} that $\psi_{\zeta(\omega)}= \Phi_\omega\circ h$ in a neighbourhood of $\si_j$, and consequently, for such $\omega$,
\be\label{third}
A\psi_{\zeta(\omega)}(\si_j)= A(\Phi_\omega \circ h)(\si_j).
\ee
Each of the equations \eqref{first},  \eqref{third} and \eqref{Apsizo} hold for $\omega$ in a cofinite subset of $\t$. Hence, for $\omega$ in the intersection of these cofinite subsets,
\[
Ap(\si_j) = 2A(\Phi_\omega\circ h)(\si_j)= 2A\psi_{\zeta(\omega)}(\si_j) = 2\rho_j
\]
as required.   \qed
\end{proof}
\begin{remark}  {\rm Under the assumptions of Theorem {\rm \ref{thm7.20}}, 
the condition \eqref{eq7.220}:
\be
(2p_0c- s_0d)^2 \neq (-2p_0a+s_0b)(s_0c -2d)
\ee
is satisfied automatically and the
rational $\Ga$-inner function $h$ is such that 
$h(\d^-)$ is not in the royal variety $ \mathcal{R}$.
}
\end{remark}

%

\begin{remark} \label{every_solution} {\rm {\em Every} solution of a royal $\Ga$-interpolation problem is obtainable by the method in the theorem.
Let data $(\si,\eta,\rho)$ be as in Theorem {\rm \ref{thm7.20}}. 
Suppose that Problem {\rm \ref{prob}} with these data is solvable and 
the solutions $\ph$ of Problem {\rm \ref{prob}}  have normalized parametrization
\[
\ph= \frac{a\zeta+b}{c\zeta+d}.
\]
By Theorem \ref{thm7.10}, 
every rational $\Ga$-inner function $h=(s,p)\in \rnk$ satisfying

{\rm (i)} $h(\sigma_j) =(-2\eta_j, \eta^2_j)$ for  $j=1,2,\ldots,n$,

{\rm (ii)} $  Ap(\sigma_j)=2\rho_j$ for $j=1,2,\ldots,k$\\
is expressible by the equations
\begin{align}
\label{eq7.110bis}
s&=2\frac{2\p0 c - \s0 d}{\s0 c - 2d},\\
\label{eq7.130bis}
p&=\frac{-2\p0 a + \s0 b}{\s0 c - 2d}
\end{align}
for some choice of $s_0, p_0$ satisfying conditions \eqref{s0p0bGa} and \eqref{eq1.801}.
}
\end{remark}

\begin{example} {\rm   Consider $3$ distinct points $\sigma_1,\si_2,\si_3 \in \t $ and $3$ distinct points $\eta_1,\eta_2,\eta_3 \in \t $  in the same cyclic order as $\sigma_1, \sigma_2, \sigma_3$.  There is a Blaschke factor $\ph $ such that
$\ph(\sigma_j) = \eta_j$ for $j =1,2,3$; let $\rho_j = A\ph(\si_j) $ for $j =1,2,3$.    Problem \ref{prob} with data $(\si,\eta,\rho)$ is solvable and $\ph$ is a solution.  Let $h =(-2\ph, \ph^2)$; then  $h(\d)\subset \mathcal{R}$. 
Every point of $\d^-$ is a royal node of $h$; in particular, $h$  has the $3$ distinct royal nodes $\si_1,\si_2,\si_3 \in \t$ with corresponding royal values $\eta_1,\eta_2,\eta_3 \in \t$, and
\[
Ap(\si_j) = A \ph^2(\si_j)=2A \ph(\si_j)=2\rho_j, \; \qquad j =1,2,3.
\]
At the same time $\deg(h)=2$. 
The example shows that for the rational $\Gamma$-inner functions whose range is contained in  $\mathcal{R}$, it can happen that $\deg(h)$ is strictly less than $n$.
}
\end{example}

\section{The algorithm}\label{algorithm}
In this section we summarize the steps in the solution of the royal $\Ga$-interpolation problem in the form of a concrete algorithm.

We suppose given Blaschke interpolation data $(\si,\eta,\rho)$ as in Definition \ref{blaschkedata}.  Here there are $n$ prescribed royal nodes $\si_j$, of which the first $k$ lie in $\t$ and the remaining $n-k$ are in $\d$.
To construct a rational $\Ga$-inner function or functions of degree $n$ having royal nodes $\si_j$, royal values $\eta_j$ and phasar derivatives $2\rho_j$ at $\si_j$ we proceed as follows.

\noindent (1) Form the Pick matrix $M=[m_{ij}]_{i,j=1}^n$ for the data $(\si,\eta,\rho)$, with entries
$$
m_{ij} = \twopartdef{\rho_i}{\mbox{ if }i=j\leq k}{\ds \frac{1-\overline{\eta_i}\eta_j}{1-\overline{\sigma_i}\sigma_j}}{\mbox{ otherwise}.}
$$
If $M$ is not positive definite then the interpolation problem is not solvable.  Otherwise, introduce the notation
\be\label{x-y-lambdabis}
x_\lambda =\begin{bmatrix}\frac{1}{1-\overline{\sigma_1}\lambda}\\ \vdots\\ \frac{1}{1-\overline{\sigma_n}\lambda}\end{bmatrix}, \qquad y_\lambda =\begin{bmatrix}\frac{\overline{\eta_1}}{1-\overline{\sigma_1}\lambda}\\ \vdots\\ \frac{\overline{\eta_n}}{1-\overline{\sigma_n}\lambda}\end{bmatrix},
\ee
as in equations \eqref{x-y-lambda}.

\noindent (2) Choose a point $\tau\in\t\setminus\{\si_1,\ldots,\si_k\}$ such that the set of $\zeta \in\t$ for which 
\[
\ip{M^{-1}x_\tau}{e_j} = \zeta\ip{M^{-1}y_\tau}{e_j} \quad\mbox{for some }j\in \{1,\ldots,n\}
\]
(where $e_j$ is the $j$th standard basis vector in $\c^n$) is finite. 

\noindent (3)  Find $s_0,p_0\in\c$ such that $|s_0|<2,\, |p_0|=1, \, s_0=\bar s_0 p_0$ and, for all $\la\in\d$,
\be\label{s0p0}
s_0\left(\ip{x_\la}{M^{-1}x_\tau}+\ip{y_\la}{M^{-1}y_\tau}\right)+2\ip{x_\la}{M^{-1}y_\tau}+2p_0\ip{y_\la}{M^{-1}x_\tau}=0.
\ee
If there is no pair $(s_0,p_0)$ satisfying these conditions, then the interpolation problem is not solvable; otherwise

\noindent (4) Let
\[
g(\la)=\prod_{j=1}^n\frac{1-\bar\si_j\la}{1-\bar\si_j\tau}
\]
and let polynomials $a,b,c,d$ be given by
\begin{align*}
a(\la)&= g(\la)\left(1-(1-\bar\tau\la)\ip{x_\la}{M^{-1}x_\tau}\right), \\
b(\la)&= g(\la) (1-\bar\tau\la)\ip{x_\la}{M^{-1}y_\tau},\\
c(\la)&= -g(\la)(1-\bar\tau\la)\ip{y_\la}{M^{-1}x_\tau},\\
d(\la)&= g(\la)\left(1+(1-\bar\tau\la)\ip{y_\la}{M^{-1}y_\tau}\right).
\end{align*} 
Note that
\[
\bbm a(\tau)&b(\tau) \\ c(\tau)&d(\tau) \ebm  = \bbm 1 & 0\\0 & 1 \ebm.
\]

\noindent (5)
Let
\begin{align*}
s&=2\frac{2\p0 c - \s0 d}{\s0 c - 2d},\\
p&=\frac{-2\p0 a + \s0 b}{\s0 c - 2d}.
\end{align*}
It is easy to see that
$$
s(\tau)= s_0 \;\; \text{and}\;\; p(\tau)= p_0.
$$
Then $h=(s,p)$ is a rational $\Ga$-inner function of degree at most $n$ such that $h(\si_j)=(-2\eta_j,\eta_j^2)$  for $j=1,\dots,n$ and $Ap(\si_j)=2\rho_j$ for $j=1,\dots,k$. By assumption, $|\p0| =1$ and $|s_0|<2$, and  hence $s(\tau)^2 \neq 4p(\tau)$.
Therefore $h(\d^-)$ is not in the royal variety $ \mathcal{R}$ and the degree of $h$ is exactly $n$.

The following comments relate the steps of the algorithm to results in the paper.

\noindent (1)  If the royal $\Ga$-interpolation problem is solvable, then the Blaschke interpolation problem with the same data is solvable, by Theorem \ref{thm7.10}.  By Proposition \ref{M>0}, $M>0$.

\noindent (2) This amounts to saying that $Z_\tau$ is finite, in the notation of equation \eqref{defZtau}. By Proposition \ref{Ztau}, there are uncountably many such $\tau\in\t$.

\noindent (3)  The necessity of the existence of $s_0,p_0$ is given in Theorem \ref{thm7.10} equation \eqref{eq7.200}, together with the equations \eqref{eq6.90} to \eqref{abcd} for $a,b,c$ and $d$.

The conditions that $|s_0|<2,\, |p_0|=1$ and $ s_0=\bar s_0 p_0$ are equivalent to $(s_0,p_0)\in b\Ga$ and $|s_0|<2$. By a standard parametrization of $b\Ga$ \cite[Theorem 2.4]{AY04}, we can take $s_0=2t \omega, \, p_0=\omega^2$ for some $t\in(-1,1)$ and $\omega\in\t$.  The condition \eqref{s0p0} then becomes: for all $\la\in\d$,
\be\label{stringent}
\ip{y_\la}{M^{-1}x_\tau}\omega^2 +t\left(\ip{x_\la}{M^{-1}x_\tau}+\ip{y_\la}{M^{-1}y_\tau}\right)\omega+\ip{x_\la}{M^{-1}y_\tau} =0.
\ee
After multiplication of both sides by $\prod_{j=1}^n (1-\bar\si_j\la)$, the coefficients in this equation relating $t$ and $\omega$ become polynomials in $\la$ of degree at most $n$, and so the equation is in effect a system of $2n+2$ real equations in two real variables.  Consequently the system is over-determined.  The existence of $s_0,p_0$ satisfying equations \eqref{s0p0} is thus in principle a stringent condition for the solvablility of a royal $\Ga$-interpolation problem.
Remarkably, in the two examples in the next section, the $\la$ terms factor out entirely from equation \eqref{stringent}, and one obtains a single real equation in $t$ and $\omega$, which has a $1$-parameter family of solutions.

\noindent (4)  The equations for $a,b,c$ and $d$ are equations \eqref{eq6.90} to \eqref{abcd}.

\noindent (5)
The equations for $s$ and $p$ are \eqref{eq7.110} and \eqref{eq7.130}.

\section{Two examples}\label{examples}
Even the simplest case of Problem \ref{royalinterp}, the royal $\Ga$-interpolation problem with only one interpolation node, demands a surprising amount of calculation to solve.  This problem is so simple that it can be readily solved without the foregoing theory, but it is instructive to see how the algorithm in Section \ref{algorithm} works in this case.
\begin{example}\label{simplest} \rm
Consider the case $n=1,\, k=0$ of Problem \ref{royalinterp}. There are prescribed a single royal node $\si_1\in\d$ and a single royal value $\eta\in\d$, and we seek a $\Ga$-inner function $h$ of degree $1$ such that $h(\si_1)=(-2\eta,\eta^2)$.  By composition with an automorphism of $\d$ we may reduce to the case that $\si_1=0$.  There is clearly at least a $1$-parameter family of solutions, if any, since if $h$ is a solution then so is $h(\omega\la)$ for any $\omega\in\t$.

The recipe for $h$ in Section \ref{sec7} proceeds as follows.  Choose an arbitrary $\tau\in\t$. 
The normalized parametrization of the solution set of the associated Blaschke interpolation problem, according to equations \eqref{abcd}, is given by
\begin{align*}
a_\tau(\la) &= \frac{\bar\tau\la-|\eta|^2}{1-|\eta|^2}, \\
b_\tau(\la)&= \frac{\eta(1-\bar\tau\la)}{1-|\eta|^2}, \\
c_\tau(\la) &= -\frac{\bar\eta(1-\bar\tau\la)}{1-|\eta|^2}, \\
d_\tau(\la) &=\frac{1-|\eta|^2\bar\tau\la}{1-|\eta|^2}. 
\end{align*}
The next step is to determine whether there exist $s_0,p_0$ such that equations \eqref{eq7.170} to \eqref{eq7.220} hold.  A little calculation shows that there is a $1$-parameter family of such $(s_0,p_0)$, given by
\[
s_0= -\frac{4\omega \re (\bar\omega\eta)}{1+|\eta|^2}, \qquad p_0=\omega^2
\]
for any $\omega\in\t$.  Substitution of these values into equations \eqref{eq7.110} and \eqref{eq7.130} yields the degree $1$ $\Ga$-inner function
\[
h(\la)= \left(-2\frac{\eta+\bar\eta\kappa\la}{1+\bar\eta^2\kappa\la}, \frac{\kappa\la+\eta^2}{1+\bar\eta^2\kappa\la}\right)
\]
where 
\[
\kappa=\bar\tau\frac{\omega^2-\eta^2}{1-\bar\eta^2\omega^2}.
\]
$\kappa$ is a general point of $\t$, and so we do obtain a $1$-parameter family of $\Ga$-inner functions of degree $1$ satisfying $h(0)=(-2\eta,\eta^2)$.  An alternative expression for $h$ is
\[
h(\la)=(\beta+\bar\beta p(\la),p(\la))
\]
where
\[
\beta= -\frac{2\eta}{1+|\eta|^2}, \qquad p(\la)=\frac{\kappa\la+\eta^2}{1+\bar\eta^2\kappa\la}.
\]
\end{example}
\begin{example}\label{n=k=1} \rm
Next consider the case of a single interpolation node on the unit circle -- say $\si=1$.  A point $\eta\in\t$ and a $\rho>0$ are prescribed, and we seek a $\Ga$-inner function $h=(s,p)$ of degree $1$ such that $h(1)=(-2\eta,\eta^2)$ and $Ap(1)=2\rho$.

Choose $\tau\in\t\setminus\{1\}$.
Again calculate the normalized parametrization of the solution set of the associated Blaschke interpolation problem according to equations \eqref{abcd}:
\begin{align*}
a_\tau(\la) &= \frac{1-\la}{1-\tau}-\frac{1-\bar\tau\la}{\rho|1-\tau|^2}, \\
b_\tau(\la)&= \frac{\eta(1-\bar\tau\la)}{\rho|1-\tau|^2}, \\
c_\tau(\la) &= -\frac{\bar\eta(1-\bar\tau\la)}{\rho|1-\tau|^2}, \\
d_\tau(\la) &=\frac{1-\bar\tau\la}{\rho|1-\tau|^2}+ \frac{1-\la}{1-\tau}.
\end{align*}
Equations \eqref{eq7.170} to \eqref{eq7.220} for $(s_0,p_0)$ have solution
\[
s_0=-\eta-\omega^2\bar\eta, \qquad p_0= \omega^2
\]
for any $\omega\in\t\setminus\{\eta\}$. 
Then equations \eqref{eq7.110} and \eqref{eq7.130} yield the degree $1$ $\Ga$-inner function
$h=(s,p)$ where
\begin{align} \label{gotsandp}
s(\la)&= 2\frac{\rho(\eta+\omega^2\bar\eta)(1-\bar\tau)(1-\la)+(\eta-\omega^2\bar\eta)(1-\bar\tau\la)}{-2\rho(1-\bar\tau)(1-\la)+(\omega^2\bar\eta^2-1)(1-\bar\tau\la)},  \notag  \\
p(\la)&= \frac{-2\omega^2\rho(1-\bar\tau)(1-\la)+(\omega^2-\eta^2)(1-\bar\tau\la)}{-2\rho(1-\bar\tau)(1-\la)+(\omega^2\bar\eta^2-1)(1-\bar\tau\la)}.
\end{align}
One can check directly that $h=(s,p)$ is a $\Ga$-inner function of degree $1$ satisfying $h(1)=(-2\eta,\eta^2)$ and $Ap(1)=2\rho$.  It appears at first sight that we have constructed a $2$-parameter family of functions with the prescribed royal node, value and phasar derivative, since the parameters $\omega$ and $\tau$ range through $\t$ (or at least, cofinite subsets thereof).  However, by means of some entertaining algebra, one can express $h$ in terms of a single unimodular parameter (the same thing happened, though more simply, in Example \ref{simplest}).
Let
\[
\kappa=\tau\frac{2\rho(1-\bar\tau)\omega-\bar\tau(\omega-\bar\omega)}{2\rho(1-\tau)\bar\omega-\tau(\bar\omega-\omega)}.
\]
Clearly $\kappa$ is unimodular.  Now let
\[
\al(\kappa)= \frac{2\rho-\bar\kappa}{1+2\rho}.
\]
It transpires that $|\al(\kappa)|<1$ and
\[
\al(\kappa)= \frac{2\rho(1-\bar\tau)-1+\bar\omega^2}{2\rho(1-\bar\tau)-\bar\tau(1-\bar\omega^2)}.
\]
One may verify that the functions $s$ and $p$ in equations \eqref{gotsandp} can be written
\begin{align} \label{againsp}
s(\la)&=-\eta-\bar\eta p(\la), \notag\\
p(\la) &= \eta^2\kappa  \frac{\la-\al(\kappa)}{1-\overline{\al(\kappa)}\la},
\end{align}
with $\kappa\in\t$, evidently a $1$-parameter family.

It is noteworthy that the function $h=(s,p)$ defined by equations \eqref{againsp} maps $\d$ into the disc $\{(\eta+\bar\eta z,z):z\in\d)\}$, which is a subset of the topological boundary $\partial\Ga$ of $\Ga$.  Inner functions $h$ such that $h(\d)\subset \partial\Ga$ were called {\em superficial} in \cite{ALY12}  and discussed in \cite[Proposition 8.3]{ALY12}.  The example shows that the solutions of a royal $\Ga$-interpolation problem can be superficial.
\end{example}

\section{Concluding remarks}\label{conclude}
In this section we relate the results of the paper to some classical results in the theory of invariant distances and thereby describe some of the original motivation for our work.

The algorithm which is developed in this paper provides constructions of 
$n$-extremal maps and $m$-geodesics in  $ \hol(\d,\G)$ with prescribed  royal nodes $\si_j$, royal values $\eta_j$ and phasar derivatives at $\si_j$. The $n$-extremal maps simultaneously generalize both Blaschke products and complex geodesics and constitute a significant class.

Recall that for a domain $G$ in $\c^N$ the {\em Carath\'eodory distance} $C_G$ on $G$ is defined by
\be\label{eq1.22}
C_G(z_1,z_2) = \sup_{F \in \hol(G,\d)} \rho(F(z_1),F(z_2)).
\ee
In equation \eqref{eq1.22}~ $z_1$ and $z_2$ are two points in $G$, $\rho$ denotes the pseudohyperbolic distance  on $\d$,
\[
\rho(z,w) = \left| \frac{z-w}{1- \bar{w} z} \right|.
\]
and, for any domain $G$ and any set  $E$,  $\hol(G,E)$ denotes the space of holomorphic mappings from $G$ to $E$.
A dual notion is the {\em Kobayashi distance} of $G$, which is defined to be the largest pseudodistance $K_G$ subordinate to the Lempert function $\rho_G$ of $G$ (e.g. \cite{Le86,Di89}). The {\em Lempert function} of $G$  is given by
\be\label{eq1.24}
\rho_G(z_1,z_2) = \inf_{\substack{h \in \hol(\d,G)\ \lambda_1,\lambda_2 \in \d\\ h(\lambda_1)=z_1\\ h(\lambda_2)=z_2}} \rho(\lambda_1,\lambda_2).
\ee

The {\em Kobayashi extremal problem} for a pair of points $z_1, z_2 \in  G$ is to find the quantity
$\rho_G(z_1,z_2)$ (\cite{Ko98}).
Any  function $h\in \hol(\d,G)$ for which the infimum
is attained is called a {\em Kobayashi extremal function} for the domain $G$ and the points $z_1, z_2$. 
In the special case when $G = \G$ it turns out that the $1$-parameter family  $\Phi_\omega \in \hol(\G,\d)$,  which we encountered in equation \eqref{eq1.26_I}, is ``universal" for the Carath\'eodory extremal problem \cite[Corollary 3.4]{AY04}, the following sense.
\begin{theorem}\label{thm1.10}
If $z_1, z_2 \in \G$ then there exists $\omega \in \t$ such that
\be\label{eq1.27}
C_G(z_1,z_2) = \rho(\Phi_\omega(z_1),\Phi_\omega(z_2)).
\ee
\end{theorem}
Another fact about the complex geometry of $\G$ is that
$$\rho_\G=K_\G=C_\G.$$ 
This corresponds to the geometric property of $\G$ that if $h$ is an extremal function for the Kobayashi problem \eqref{eq1.24}, then the range $\ran(h)$ of $h$ is a totally geodesic analytic disc in $\G$ \cite[Corollary 5.7]{AY04}.

The Kobayashi extremal problem can be viewed as an extremal $2$-point interpolation problem. Specifically, by a {\em finite interpolation  problem in $\hol(\d,G)$}, one means the following.
\begin{problem}\label{prob1.10} 
Given $n$ distinct points $\lambda_1,\ldots,\lambda_n$ in $\d$ and $n$ points $z_1,\ldots,z_n $ in an open or closed set $G\subset \C^N$, to determine
whether there exists a function $h \in \hol(\d,G)$ such that $h(\lambda_j) = z_j$ for $j=1,\ldots,n$.
\end{problem}
We say that Problem \ref{prob1.10} {\em  is solvable}, or that the data $\la_j \mapsto z_j$ {\em are solvable}, if there does exist an $h \in \hol(\d,G)$ that satisfies these interpolation conditions. 
We say that the problem is (or the data are)  \emph{extremal} when the problem is solvable but there do not exist an open neighbourhood $U$ of the closure of $\d$ and a map $h\in\hol(U,G)$ such that the conditions 
\beq\label{interpGen1}
h(\la_j)=z_j \quad\mbox{ for } \quad  j=1,\dots,n,
\eeq
hold.

A map $h\in\hol(\d,G)$ is said to be {\em $n$-extremal} if, for any choice of $n$ distinct points $\la_1,\dots,\la_n\in\d$, the interpolation data $\la_j\in\d \mapsto h(\la_j)\in  G$ are extremally solvable.

With this perspective, if $h$ and $\lambda_1,\lambda_2$  minimize the right hand side of  equation \eqref{eq1.24}, then the $2$-point interpolation problem $\lambda_j \to z_j,\ j=1,2$ for $\hol(\d,G)$ is extremal and $h$ is an extremal solution to it.  Just as the Kobayashi extremal functions on $\G$ are both rational and $\gaminn$, more generally, the following result obtains (see \cite{Cost05} or \cite[Theorem 3.1]{ALY13}). 
\begin{proposition}\label{prop1.10}
If $\lambda_j \to (s_j,p_j)$,  $j=1,\ldots,n$,  is a solvable $n$-point interpolation problem for $\hol(\d,\G)$ then it has a rational $\gaminn$ solution.
\end{proposition}

The royal variety (or more precisely, $\royal \cap\G$) is a complex geodesic of $\G$, with extremal function given by $h_\royal(\lambda)=(2\lambda,\lambda^2)$. Furthermore, among the complex geodesics in $\G$, the royal variety is characterized by the property that 
$$
\alpha(\royal\cap\G)=\royal\cap\G
$$ 
whenever $\alpha$ is a biholomorphic self map (automorphism) of $\G$ \cite[Lemma 4.3]{AY08}. In addition, the automorphism group of $\G$ acts transitively on $\royal\cap \G$.

If $F$ is a Carath\'eodory extremal function for some pair of points in $\G$ then so is $m\circ F$ for any M\"obius transformation $m$ of the disc. The universal set described in Theorem \ref{thm1.10} above is normalized so as to satisfy
$\Phi_\omega \circ h_\royal =-\mathrm{ id}_\d$. As a result,
\be\label{eq1.40}
\Phi_\omega|\royal \text{ does not depend on } \omega.
\ee

A Kobayashi extremal function on any domain for which the Lempert function and the Carath\'eodory distance coincide has a holomorphic left inverse.  L. Kosinski and W. Zwonek \cite{KZ} introduced a generalization of this notion: a map $h: \d\to G$, for any domain $G$, is said to be an {\em $n$-complex geodesic} if there exists a holomorphic map $F:G\to\d$ such that $F\circ h$ is a Blaschke product of degree at most $n$.
  The following result shows that rational  $\Gamma$-inner functions enjoy this property.
\begin{proposition}\label{extrem-maps} Let $h$ be a rational  $\Gamma$-inner function of degree $n$ which is not superficial and let $h(\d)\not\subset\royal$.   Then 
\begin{enumerate}
\item $h$ is an $(n+1)$-extremal holomorphic map in  $\hol(\d,\Gamma)$ and is an $n$-complex geodesic of $\G$;
\item if in addition $h$ has at least one royal node $\sigma \in \t$ then $h$ is an $n$-extremal holomorphic map in  $\hol(\d,\Gamma)$ and is an $(n-1)$-complex geodesic of $\G$.
\end{enumerate}
\end{proposition}
\begin{proof} As in Theorem \ref{thm7.10}, $\ph = \Phi_\omega\circ h$ for some $\omega \in \t$ is 
 a rational inner function $\ph \in \hol(\d,\d^-)$  such that
 $\deg(\ph) \le n$. Thus 
$h$ is $n$-complex geodesic. By a version of Pick's result, the $(n+1)$-extremal holomorphic self-maps of $\d$ are precisely the Blaschke products of degree at most $n$. Thus  $\ph$ is a $(n+1)$-extremal in  $\hol(\d,\d^-)$. Therefore, by \cite[Proposition 2.2]{ALY12}, $h$ is 
$(n+1)$-extremal in  $\hol(\d,\Gamma)$.

(ii) If $h(\sigma) = (-2\eta, \eta^2)$ and
 $\omega=-\bar\eta$ then the rational function $\Phi_\omega\circ h$ has a removable singularity at $\si$.  After cancellation  $\Phi_\omega\circ h$ has degree $(n-1)$. As above $h$ is an $n$-extremal holomorphic map in  $\hol(\d,\Gamma)$ and is an $(n-1)$-complex geodesic of $\G$.
\qed
\end{proof}

\begin{corollary}\label{geodesics} All non-superficial functions $h$ in $\mathcal{R}^{1,0} \cup \mathcal{R}^{2,1}$ are complex geodesics of $\G$.
\end{corollary}
\begin{proof}  First we recall a result from \cite{AY06} that  an analytic  function $h:\D \to \G  $ is  a  complex geodesic of $\G$ if and only if there is an $\omega \in \T$ such that $\Phi_{\omega} \circ h \in\Aut \mathbb{D}$ and that  every complex geodesic of $\G$ is $\Gamma$-inner.
By Proposition \ref{extrem-maps}, each non-superficial function from the set $\mathcal{R}^{1,0} \cup \mathcal{R}^{2,1}$ is a complex geodesic. \qed
\end{proof}


\vspace*{1cm}

JIM  ~ AGLER, Department of Mathematics, University of California at San Diego, CA \textup{92103}, USA\\

ZINAIDA A. LYKOVA,
School of Mathematics and Statistics, Newcastle University, Newcastle upon Tyne
 NE\textup{1} \textup{7}RU, U.K.~~\\

N. J. YOUNG, School of Mathematics and Statistics, Newcastle University, Newcastle upon Tyne NE1 7RU, U.K.
{\em and} School of Mathematics, Leeds University,  Leeds LS2 9JT, U.K.

\begin{thebibliography}{50} \label{bibliog}


\bibitem{ALY12}  J. Agler, Z. A. Lykova and N. J. Young, Extremal holomorphic maps and the symmetrized bidisc,  {\em  Proc. London Math. Soc.}  {\bf 106}(4) (2013)  781-818.

\bibitem{ALY13} J. Agler, Z. A. Lykova and N. J. Young,
A case of $\mu$-synthesis as a  quadratic semidefinite program, {\em SIAM Journal on Control and Optimization}, {\bf 51}(3) (2013) 2472-2508.

\bibitem{ALY14} J. Agler, Z. A. Lykova and N. J. Young, Algebraic and geometric aspects of rational $\Gamma$-inner functions, arXiv:1502.04216

\bibitem{AgMcC} J. Agler and J. McCarthy, {\em Pick Interpolation and Hilbert Function Spaces}, Graduate Studies in Mathematics {\bf 44},  Amer. Math. Soc., Providence, R.I. 2002. 
\bibitem{AY99} J. Agler and N. J. Young, A commutant lifting theorem for a domain in $\C^2$ and spectral interpolation, 
{\em J. Funct. Anal.} \textbf{161} (1999) 452--477.

\bibitem{AY} J. Agler and N. J. Young, Operators having the symmetrized bidisc as a spectral set, {\em Proc. Edin. Math. Soc.} {\bf 43} (2000) 195-210.

\bibitem{AY04}  J. Agler and N. J. Young, The hyperbolic geometry of 
the symmetrized bidisc, {\em J. Geom. Anal.} {\bf 14} (2004) 375--403.

\bibitem{AY06}  J. Agler and N. J. Young, The complex geodesics of the symmetrized bidisc, {\em Inter. J. of Mathematics} {\bf 17}, no.4, (2006) 375--391.

\bibitem{AY08} J. Agler and N. J. Young, The magic functions and automorphisms of a domain, {\em Complex Analysis and Operator Theory} {\bf 2} (2008) 383-404.



\bibitem{bgr}   J. A. Ball, I. Gohberg and L. Rodman, \emph{Interpolation of Rational Matrix Functions}, Operator Theory: Advances and Applications \textbf{45} (Birkh\"auser Verlag, Basel, 1990).



\bibitem{BH}
J. A. Ball and J. W. Helton, Interpolation problems of Pick-Nevanlinna and Loewner types for meromorphic matrix functions: parametrization of the set of all solutions. 
\emph{Integral Equ. Oper. Theory} \textbf{9} (1986) 155--203.


\bibitem{beurling} A. Beurling, On two problems concerning linear transformations in Hilbert space, {\em Acta Math.} {\bf 81} (1949) pp. 239--255. 

\bibitem{Ber03} H. Bercovici, Spectral versus classical Nevanlinna-Pick interpolation in dimension two,
{\em Electronic Journal of Linear Algebra}
{\bf 10} (2003), 60--64.

\bibitem{bhattacharyya} T. Bhattacharyya, S. Pal and S. Shyam Roy, Dilations of $\Gamma$-contractions by solving operator equations, {\em Adv. in Math.}
\textbf{230} (2012), 577–-606.

\bibitem{bhattacharyya2} T. Bhattacharyya and S. Pal,
A functional model for pure $\Gamma$-contractions, {\em J. Operator Theory} \textbf{71} (2014), 327--339.

\bibitem{biswas} S. Biswas and S. Shyam Roy, Functional models of $\Gamma_n$-contractions and characterization of $\Gamma_n$-isometries, {\em J. Funct. Anal.} \textbf{ 266} (2014), no. 10, 6224--6255.  

\bibitem{blaschke} W. Blaschke, Eine Erweiterung des Satzes von Vitali $\ddot{\rm u}$ber Folgen analytischer Funktionen, {\em Berichte Math.-Phys. Kl., S$\ddot{\rm a}$chs. Gesell. der Wiss. Leipzig}, {\bf 67} (1915)  194--200.


\bibitem{bol2011} V. Bolotnikov, On zeros of certain analytic functions, {\em Integral Equations Operator Theory} {\bf 69} (2011), no. 2, 203-215.

\bibitem{BD}   V. Bolotnikov and H. Dym.  On boundary interpolation for matrix valued Schur functions. \emph{AMS Memoirs} \textbf{181} (2006)  1--107.

\bibitem{ChH}  G.-N. Chen and Y.-J. Hu. Multiple Nevanlinna-Pick interpolation with both interior and boundary data and its connection with the power moment problem. \emph{Linear Algebra Appl.} \textbf{323} (2001) 167--194.

\bibitem{costara} C. Costara, The symmetrized bidisc and Lempert's theorem, {\em Bull. London Math. Soc.} {\bf 36} (2004) 656--662. 
\bibitem{Cost05} C. Costara, On the spectral Nevanlinna-Pick problem, {\em Studia Math.} {\bf 170} (2005) 23--55.

\bibitem{Di89} S. Dineen, {\em The Schwarz Lemma}, Oxford University Press, 248 pages, (1989).

\bibitem{edigarian} A. Edigarian, Balanced domains and convexity, {\em Arch. Math. (Basel) }  {\bf 101} (2013), no. 4, 373--379.

\bibitem{geo}  {  D. R. Georgijevi\'c}, Mixed L$\ddot{\rm o}$wner and Nevanlinna-Pick Interpolation,  \emph{Integral Equations and Operator Theory}  \textbf{53} (2005) 247--267. 

\bibitem{gl2002} C. Glader and M. Lindstr$\ddot{\rm o}$m, Finite Blaschke product interpolation
on the closed unit disc, {\em  J. Math. Anal. Appl.} {\bf 273} (2002) 417-427.

\bibitem {gr2008} P. Gorkin and R. C. Rhoades, Boundary interpolation by finite Blaschke products, {\em Constr. Approx.} {\bf 27} (2008), no. 1, 75-98.

\bibitem{jarnicki} M. Jarnicki and P. Pflug, On automorphisms of the symmetrized bidisc, {\em Arch. Math.} (Basel) {\bf 83} (2004), no. 3, 264--266.

\bibitem {jr87} W. B. Jones and S. Ruscheweyh, Blaschke product interpolation and
its application to the design of digital filters, {\em Constr. Approx.} {\bf  3} (1987) 405-409.

\bibitem {kh1998} A. Kheifets, The abstract interpolation problem and applications, in {\em Holomorphic spaces}, Math. Sci. Res. Inst. Publ. {\bf 33} (1998)  351-379, Cambridge Univ. Press, Cambridge.

\bibitem{Ko98} S. Kobayashi, {\em Hyperbolic complex spaces}, Springer, New York, 1998.

\bibitem{KZ}  L. Kosi\'nski and W. Zwonek, Extremal holomorphic maps in special classes of domains, {\em Annali Scuola Normale Superiore di Pisa - Science Class}  {\bf 16} (2016) 159-182.



\bibitem{Le86} L. Lempert, Complex geometry in convex domains, {\em Proc. Intern. Cong. Math.}, Berkeley, CA (1986) 759--765.

\bibitem{NiPfZw}
N. Nikolov, P. Pflug and W. Zwonek, The Lempert function of the symmetrized polydisc in higher dimensions is not a distance,
 {\em Proc. Amer. Math. Soc.} {\bf 135} (2007) 2921--2928.

\bibitem{NiPa}
N. Nikolov, P. J. Thomas and T. Duc-Anh, Lifting maps from the symmetrized polydisk in small dimensions, 
{\em Complex Anal. Oper. Theory}  {\bf 10} (2016), no. 5, 921--941.

\bibitem{pal} S. Pal and O. M. Shalit, Spectral sets and distinguished varieties in the
symmetrised bidisc, {\em J. Funct. Anal.}  {\bf 266} (2014), no. 9, 5779--5800. 

\bibitem{pal2} S. Pal, From Stinespring dilation to Sz.-Nagy dilation on the symmetrized bidisc and operator models, {\em New York J. Math.}
{\bf 20} (2014), 545--564.

\bibitem{PZ} P. Pflug and  W. Zwonek, Description of all complex geodesics in the symmetrised bidisc, {\em Bull. London Math. Soc.} {\bf 37} (2005) 575--584.

\bibitem{RR} M. Rosenblum and J. Rovnyak, {\em Hardy classes and operator theory},  Oxford University Press, New York, 1985. xiv+161 pages.

\bibitem{Sar}    D. Sarason. Nevanlinna-Pick interpolation with boundary data, \emph{Integr. Equ. Oper. Theory} \textbf{30} (1998) 231--250.

\bibitem{sarkar} J. Sarkar, Operator theory on symmetrized bidisc, {\em Indiana Univ. Math. J.} \textbf{64} (2015), no. 3, 847--873.

\bibitem{sw2006} G. Semmler and E. Wegert, Boundary interpolation with Blaschke products of minimal degree, {\em Comput. Methods Funct. Theory} {\bf 6} (2006), no. 2, 493-511.

\bibitem{NJY11} N. J. Young,  Some analysable instances of $\mu$-synthesis, in {\em Mathematical methods in systems, optimization and control}, Editors: H. Dym, M. de Oliveira, M. Putinar, Operator Theory: Advances and Applications {\bf 222} (2012) 349--366, Springer Verlag, Basel.


\end{thebibliography}
\end{document}